%% file: circle_equivariant_tate_filtration.tex
\documentclass[oneside,a4paper]{amsart}

\input{preamble}
\usepackage{todonotes}

\input{macros}

\begin{document}

\title{On the circle-equivariant cellular Tate filtration}
\author{Toni Annala}
\author{Piotr Pstr\k{a}gowski}

\begin{abstract}
Using an argument of Jacob Lurie, we prove the existence of a $\mathbf{E}_{2}$-lax monoidal structure on the Tate filtration coming from the standard cell structure on $\mathrm{B}S^{1} \simeq \mathbf{CP}^{\infty}$. We then compare it to the filtration induced by the standard t-structure on spectra. In the last part of the paper, we construct a synthetic variant of the cellular Tate filtration and use it to compare Antieau's, Bhatt--Lurie's and Raksit's HKR filtrations on negative cyclic and periodic homology. 
\end{abstract}

% \date{\today}

\maketitle

% %\setcounter{tocdepth}{3}
\tableofcontents

\section{Introduction} 

Associated to a spectrum $X$ with an $S^1$-action we have the Tate construction
\[
X^{t S^{1}} \colonequals \mathrm{cofib}(\mathrm{Nm} \colon \Sigma X_{hS^{1}} \rightarrow X^{hS^{1}}) 
\]
defined as the cofibre of the norm map. The standard cell structure on $\mathrm{B}S^{1} \simeq \mathbf{CP}^{\infty}$ allows one to construct a filtered spectrum 
\[
\ldots \rightarrow F_{T}^{1}(X^{tS^{1}}) \rightarrow F_{T}^{0}(X^{tS^{1}}) \rightarrow F_{T}^{-1}(X^{tS^{1}}) \rightarrow \ldots,
\]
which we will refer to as the \emph{cellular Tate filtration}, with the properties that 
\begin{enumerate}
    \item the colimit $\varinjlim F_{T}^{\ast}(X^{tS^{1}})$ can be identified with $X^{tS^{1}}$, 
    \item the associated graded spectrum can be identified
    \[
    \gr_{T}^{\ast}(X^{tS^{1}}) \simeq X[-2 \ast] 
    \]
    with the collection of even shifts of $X$. 
\end{enumerate}
Each term in the filtration can be conveniently expressed in terms of fixed points of $X$ twisted by a suitable representation sphere. Concretely, for each $q \in \ZZ$ we have 
\begin{equation}
\label{equation:formula_for_the_tate_construction}
F_{T}^{q}(X^{tS^{1}}) \colonequals (X \otimes \thesphere^{-(q) \mathbf{C}})^{h S^{1}}, 
\end{equation}
where we view $S^{1}$ as acting on $\mathbf{C}$ through the identification $S^{1} \simeq U(1)$. The maps making this graded spectrum into a filtered object are induced by the inclusion $0 \hookrightarrow \mathbf{C}$ of the zero representation. 

In \cite[{6.1.8}]{bhatt-lurie:apc} it is observed that the Tate filtration can be made lax monoidal, although not symmetric monoidal. The purpose of this short note is to fill the gap in the literature by recording the proof of this result, which was explained to us by Jacob Lurie. In more detail, we prove the following: 

\begin{theorem}
\label{theorem:introduction_bt_tate_filtration_is_lax_e2_monoidal}
The Tate filtration functor $F_{T}^{\ast}({-}^{tS^{1}}) \colon \spectra^{BS^{1}} \rightarrow \Fil(\spectra)$ defined by (\ref{equation:formula_for_the_tate_construction}) can be made lax $\mathbf{E}_{2}$-monoidal. 
\end{theorem}

The above result is optimal in the sense that $F_{T}^{\ast}({-}^{tS^{1}})$ \emph{cannot} be equipped with an $\mathbf{E}_{3}$-monoidal structure, see  \cref{warning:tate_filtration_cannot_be_made_e3_monoidal}. In fact, such an $\mathbf{E}_{3}$-monoidal structure does not exist even in the $\mathbf{Z}$-linear context. On the other hand, the cellular Tate filtration \emph{can} be made lax symmetric monoidal over the rationals $\mathbf{Q}$ by a result of Bals \cite[{Lemma B.10}]{BalsPeriodicCyclicHomology}. 

We note that there is one more commonly employed filtration on the Tate construction, given by $(\tau_{\geq \ast} X)^{tS^{1}}$, the Tate construction applied to the Whitehead tower of $X$. This filtration, which we will refer to as the \emph{homotopy Tate filtration}, turns out to be closely related to the cellular Tate filtration, as we now explain. 

Let $d \colon \ZZ \rightarrow \ZZ$ be the multiplication-by-two functor, and let $d_{!} \colon \Fil \spectra \rightarrow \Fil \spectra$ be its left Kan extension. Then $d_{!}$ is a fully faithful embedding, whose essential image is given by those filtered spectra which are even in the sense of \cite[{\S 5.3.1}]{van2025introduction}; that is, such that their associated graded vanishes in odd degrees. We write $\Dec \colon \Fil\spectra \rightarrow \Fil\spectra$ for the d\'{e}calage functor of \cite{antieau:decalage}.

\begin{theorem}[{\ref{theorem:decalage_of_evenification_of_bhatt_lurie_tate_is_the_tate_of_postnikov}, \ref{corollary:decalaged_cellular_tate_is_double_of_decalaged_postnikov_tate}}]
\label{theorem:introduction_comparison_of_bt_tate_and_standard_tate}
If $X \in \spectra^{\mathrm{B}S^{1}}$, there are natural equivalences of filtered spectra: 
\begin{enumerate}
    \item $\Dec(d_{!}(F_{T}^{\ast}(X^{tS^{1}}))) \simeq (\tau_{\geq \ast} X)^{tS^{1}}$, 
    \item $d_{!} \Dec(F_{T}^{\ast}(X^{tS^{1}})) \simeq \Dec((\tau_{\geq \ast} X)^{tS^{1}})$.
\end{enumerate}
\end{theorem}
The first equivalence is a minor variation on a result of Antieau, who proved an analogous statement for homotopy fixed points in \cite[{Proposition 9.2}]{antieau:decalage}. The second one is more interesting, in that it follows from the first and a general formula 
\[
\Dec(\Dec(d_{!} X)) \simeq d_{!} \Dec(X)
\]
for the interaction of d\'{e}calage and the doubling functor $d_{!}$, which we prove in \cref{proposition:decalage_and_doubling_interaction}. Either part of \cref{theorem:introduction_comparison_of_bt_tate_and_standard_tate} implies that the spectral sequences associated to $F_{T}^{\ast}(X^{tS^{1}})$ and $(\tau_{\geq \ast} X)^{tS^{1}}$ are isomorphic after a suitable regrading which we make explicit in \cref{corollary:bt_tate_and_standard_tate_have_isomorphic_spectral_sequences}. 

An interesting aspect of \cref{theorem:introduction_comparison_of_bt_tate_and_standard_tate} is that the right-hand side of the equivalence is manifestly lax symmetric monoidal, while $F_{T}^{\ast}(-^{tS^{1}})$ is at most lax $\mathbf{E}_{2}$-monoidal. In other words, the obstruction to making the cellular Tate filtration lax symmetric monoidal vanishes after a single application of d\'{e}calage. 

Since the associated graded of the cellular Tate filtration has a particularly simple form, given by shifts of the object itself, it is often more natural than a filtration obtained from a Postnikov tower, which only makes sense in the presence of a t-structure. One such context is the theory of synthetic spectra, where a natural analogue of the cellular Tate filtration appears in the work of Antieau-Riggenbach \cite[{Proof of Lemma 2.75}]{antieau2024cyclotomic}. 

We make the construction of the synthetic Tate filtration explicit in \cref{definition:tate_filtration_in_synthetic_spectra}, and we calculate its colimit and the associated graded object. As a concrete application of this circle of ideas and \cref{theorem:introduction_comparison_of_bt_tate_and_standard_tate}, we compare the HKR filtrations on negative and periodic cyclic homology previously constructed by Antieau, Bhatt--Lurie, and Raksit \cite{antieau2019periodic, bhatt-lurie:apc, raksit:2020}. It is a folklore result that these three constructions agree, but, to our knowledge, a complete proof has not previously appeared in the literature. We provide such a proof here:

\begin{theorem}[{\ref{theorem:antieau_bhattlurie_raksit_hkr_filtrations_on_hcminus_are_equivalent}}]
\label{theorem:introduction_hkr_comparison}
Let $A$ be an animated $k$-algebra. Then there are equivalences, natural in $A$,
\[
F^{\ast}_{\AntHKR} F_T^* \HP(A / k) \simeq F^{\ast}_{\BLHKR} F_T^*\HP(A / k) \simeq F^*_T F^{\ast}_{\RakHKR} \HP(A / k) \in \BiFil(\dcat(k)), 
\]
where the second equivalence exchanges the filtrations. In particular, the three HKR filtrations on $\HC^-$ and $\HP$ coincide.
\end{theorem}
Here, the filtration $F^{\ast}_{T}F^{\ast}_{\RakHKR}$ is the synthetic Tate filtration applied to Raksit's filtered variant of the Tate construction. The result for $\HC^{-}$ is obtained by restricting the filtration to non-negative Tate degrees. For a recollection on the construction of these three filtrations, as well as the notation we use in the statement above, see \cref{%
  recollection:antieaus_hkr,%
  recollection:bhatt_lurie_hkr_filtration,%
  recollection:raksit_hkr_filtration%
}.

We note that, unlike for the ordinary variant, we do not show in this paper that the formation of the synthetic Tate filtration can be made $\mathbf{E}_{2}$-lax monoidal. We believe that the natural way to do so would be to make a certain variant of the circle-equivariant perfect even filtration lax symmetric monoidal, and considering this topic here would take us too far afield. We make a precise statement we expect in \cref{conjecture:monoidality_of_the_synthetic_tate_filtration}. 

\subsection{Sketch of the monoidality argument} 
\label{subsection:introduction_sketch_of_the_monoidality_argument}

In the rest of the introduction, we outline the argument behind \cref{theorem:introduction_bt_tate_filtration_is_lax_e2_monoidal}. Since $F_{T}^{\ast}(-^{tS^{1}})$ is defined as the composite of tensoring with the filtered $S^{1}$-spectrum 
\begin{equation}
\label{equation:introduction_filtered_e2_algebra_defining_the_bt_tate_filtration}
S^{q} \colonequals \thesphere^{(-q) \mathbf{C}}
\end{equation}
and taking homotopy fixed points, and since the latter is lax symmetric monoidal, to prove the statement it is enough to make $S^{\ast}$ into an $\mathbf{E}_{2}$-algebra object of $\Fil(\spectra^{hS^{1}})$. 

The starting point is a construction of Lurie that makes the association $n \mapsto \mathbf{CP}^{n}$ into an $\mathbf{E}_{2}$-coalgebra object of $\Fun(\mathbf{N}, \spaces)$, where we consider the latter as equipped with Day convolution. In \cref{lemma:coalgebras_in_filtered_spaces_lift_to_coalgebras_in_spaces_over_their_colimit}, we show that this can be uniquely lifted to a filtered coalgebra object in spaces over $\varinjlim \mathbf{CP}^{n} \simeq \mathbf{CP}^{\infty} \simeq \mathrm{B}S^{1}$. Passing to suspension spectra, this yields an $\mathbf{E}_{2}$-coalgebra object of $\Fun(\mathbf{N}, \spectra^{hS^{1}})$ given by 
\begin{equation}
\label{equation:introduction_filtered_coalgebra_given_levelwise_by_odd_spheres}
n \mapsto \Sigma^{\infty}_{+}(\mathrm{fib}(\mathbf{CP}^{n} \rightarrow \mathbf{CP}^{\infty})) \simeq \Sigma^{\infty}_{+} S^{2n+1},
\end{equation}
where $S^{1} \simeq U(1)$ acts on $S^{2n+1}$ by identifying the latter with the unit sphere of $\mathbf{C}^{n+1}$. 

The $\infty$-category $\Fun(\mathbf{N}, \spectra^{hS^{1}})$ is dualizable over $\spectra^{hS^{1}}$, and so admits a \emph{Serre functor} 
\[
\mathrm{Serre} \colon \Fun(\mathbf{N}, \spectra^{hS^{1}}) \rightarrow \Fun(\mathbf{N}, \spectra^{hS^{1}}). 
\]
In \cref{proposition:existence_of_right_day_for_fun_n_d_with_d_finitely_complete}, we verify that in addition to the usual left Day convolution (defined in terms of colimits), $\Fun(\mathbf{N}, \spectra^{hS^{1}})$ also admits a right Day convolution symmetric monoidal structure (defined in terms of limits). The Serre functor is symmetric monoidal when we consider the source with respect to left Day convolution and the target with right Day convolution (\cref{proposition:serre_functor_on_fun_n_c_interchanges_left_and_right_day_convolution}). 

It follows that the image of (\ref{equation:introduction_filtered_coalgebra_given_levelwise_by_odd_spheres}) is an $\mathbf{E}_{2}$-coalgebra object of $\Fun(\mathbf{N}, \spectra^{hS^{1}})$ with respect to the right Day convolution. Taking Spanier-Whitehead duals levelwise yields an $\mathbf{E}_{2}$-\emph{algebra} object of $\Fun(\mathbf{N}^{\op}, \spectra^{hS^{1}})$, now with respect to the usual (left) Day convolution. 

Finally, left Kan extending along $\mathbf{N}^{\op} \rightarrow \mathbf{Z}^{\op}$ and inverting an appropriate element in filtered homotopy yields an $\mathbf{E}_{2}$-algebra whose underlying filtered $S^{1}$-spectrum can be identified with (\ref{equation:introduction_filtered_e2_algebra_defining_the_bt_tate_filtration}). This ends the argument. 

\subsection{Conventions}

If $\ccat$ is an $\infty$-category, then a \emph{filtered object} of $\ccat$ is a functor $F^{\ast} X \colon \ZZ^{\op} \rightarrow \ccat$, where we view $\ZZ$ as a poset. The \emph{associated graded} is given by $\gr^{q}X \colonequals \mathrm{cofib}(F^{q+1}X \rightarrow F^{q}X)$. The \emph{underlying object} is the colimit $|F^{\ast} X| \colonequals \varinjlim F^{q} X$. We write $\Fil \ccat \colonequals \Fun(\ZZ^{\op}, \ccat)$ and $\BiFil \ccat \colonequals \Fun(\ZZ^{\op} \times \ZZ^{\op}, \ccat)$ for the $\infty$-categories of filtered and bifiltered objects. 

\subsection{Acknowledgements}

We would like to thank Ben Antieau, Tess Bouis, Sanath Devalapurkar and Jacob Lurie for helpful discussions related to this project. 

\section{Preliminaries} 

\subsection{Serre duality of filtered spectra}

In this subsection, we give a formula for the Serre functor of the $\infty$-category of filtered objects. 

\begin{recollection}
Let $\vcat$ be a presentably symmetric monoidal stable $\infty$-category and $\ecat \in \Mod_{\vcat}(\Pr^{L})$ be a dualizable $\vcat$-module, so that we have an evaluation morphism $\mathrm{ev} \colon \ecat \otimes_{\vcat} \ecat^{\vee} \rightarrow \vcat$. This is a morphism in $\Pr^L$, so it has a right adjoint $\mathrm{ev}^{R}$. The \emph{Serre functor} of $\ecat$ is the $\vcat$-linear cocontinuous functor
\[
\mathrm{Serre} \colon \ecat \rightarrow \ecat
\]
corresponding to $\mathrm{ev}^{R}(\mathbf{1}_{\vcat})$ under the equivalence $\ecat \otimes_{\vcat} \ecat^{\vee} \simeq \Fun^{L}_{\vcat}(\ecat, \ecat)$.
\end{recollection}

\begin{example}
\label{example:serre_functor_of_a_v_linear_presheaf_infty_category}
Let $\ccat$ be a small $\infty$-category, and let $\presheaves(\ccat; \vcat) \colonequals \Fun(\ccat^{\op}, \vcat)$ be the $\infty$-category of $\vcat$-valued presheaves. The latter is dualizable over $\vcat$, with dual $\presheaves(\ccat^{\op}, \vcat)$, and under the equivalence
\[
\presheaves(\ccat; \vcat) \otimes_{\vcat} \presheaves(\ccat^{\op}; \vcat) \simeq \presheaves(\ccat \times \ccat^{\op}; \vcat), 
\]
the right adjoint $\mathrm{ev}^{R} \colon \vcat \rightarrow \presheaves(\ccat \times \ccat^{\op}, \vcat)$ of the evaluation is given by the formula 
\[
\ev^{R}(v) \simeq v^{\map_{\ccat}(c_{2}, c_{1})}, 
\]
where $(-)^{\map_{\ccat}(c_{2}, c_{1})}$ is the cotensor, i.e. the limit of the constant diagram. It follows that the Serre functor 
\[
\mathrm{Serre} \colon \presheaves(\ccat; \vcat) \rightarrow \presheaves(\ccat; \vcat) 
\]
is the unique $\vcat$-linear functor such that the composite with the $\vcat$-linear Yoneda embedding 
\[
\begin{tikzcd}
	{\ccat} & {\presheaves(\ccat)} & {\presheaves(\ccat) \otimes_{\spaces} \vcat} & {\presheaves(\ccat; \vcat) }
	\arrow["y", from=1-1, to=1-2]
	\arrow["{y_{\vcat}}"', bend right=15, from=1-1, to=1-4]
	\arrow[from=1-2, to=1-3]
	\arrow[from=1-3, to=1-4]
\end{tikzcd}
\]
is given by $\mathrm{Serre}(y_{\vcat}(c)) \simeq \mathbf{1}_{\vcat}^{\map_{\ccat}(c, -)}$. In particular, 
\begin{equation}
\label{equation:formula_for_serre_of_v_presheaves}
\underline{\map}_{\presheaves(\ccat; \vcat)}(y_{\vcat}(d), \mathrm{Serre}(y_{\vcat}(c))) \simeq \mathbf{1}_{\vcat}^{\map_{\ccat}(c, d)} \simeq (\mathbf{1}_{\vcat}[\map_{\ccat}(c, d)])^{\vee} \simeq \underline{\map}_{\presheaves(\ccat; \vcat)}(y_{\vcat}(c), y_{\vcat}(d))^{\vee}, 
\end{equation}
where $\underline{\map}$ denotes the $\vcat$-mapping object, $\mathbf{1}_{\vcat}[\map_{\ccat}(c, d)]$ is the tensor and $(-)^{\vee}$ is the $\vcat$-monoidal dual. 
\end{example}

\begin{remark}
When $\vcat = \Mod_{A}(\spectra)$ is the $\infty$-category of modules over a commutative ring spectrum $A$, then objects of the form $y_{\vcat}(c) \in \presheaves(\ccat; \vcat)$ are compact and generate all compact objects as a thick subcategory. In this case, (\ref{equation:formula_for_serre_of_v_presheaves}) implies that for any compact $X, Y \in \presheaves(\ccat; \vcat)$ we have an equivalence of $A$-modules 
\[
\map_{\presheaves(\ccat; \vcat)}(X, \mathrm{Serre}(Y)) \simeq \map_{\presheaves(\ccat; \vcat)}(Y, X)^{\vee},
\]
where the right-hand side is the $A$-linear dual, as observed in \cite[{Remark 11.1.5.2}]{SAG}. 
\end{remark}

\begin{proposition}
\label{proposition:formula_for_the_serre_functor_of_filtered_spectra}
Let $\ccat$ be a presentably symmetric monoidal stable $\infty$-category. Then the Serre functor of $\Fil(\ccat) = \presheaves(\mathbf{Z}; \ccat)$ is given by the formula 
\[
(\mathrm{Serre}(X))_{q} \simeq \mathrm{cofib}(X_{q+1} \rightarrow X_{-\infty}), 
\]
where $X_{-\infty} \colonequals \varinjlim X$. 
\end{proposition}

\begin{proof}
As a $\ccat$-linear $\infty$-category, $\Fil(\ccat)$ is generated by objects of the form $y_{\ccat}(n)$ where $n \in \mathbf{Z}$, explicitly given by 
\[
y_{\ccat}(n)_{q} \colonequals
    \begin{cases}
      \mathbf{1}_{\ccat}, & \text{if } q \leq n \\
      0, & \text{otherwise}
    \end{cases}
    \]
Using (\ref{equation:formula_for_serre_of_v_presheaves}) we see that 
\[
\mathrm{Serre}(y_{\ccat}(n))_{q} \simeq \underline{\map}(y_{\ccat}(q), \mathrm{Serre}(y_{\ccat}(n))) \simeq \underline{\map}(y_{\ccat}(n), y_{\ccat}(q))^{\vee} \simeq \begin{cases}
      {\mathbf{1}_{\ccat}}, & \text{if } q \geq n \\
      0, & \text{otherwise}
    \end{cases}, 
\]
Thus, the formula in the statement holds for $y_{\ccat}(n)$. As both sides are $\ccat$-linear, we deduce that it holds in general. 
\end{proof} 

\begin{remark}
\label{remark:serre_functor_of_non_positively_filtered_objects}
By the same argument as in the proof of \cref{proposition:formula_for_the_serre_functor_of_filtered_spectra}, an analogous formula holds for the Serre functor of $\presheaves(\mathbf{Z}_{\leq 0}; \ccat)$. More precisely, for each $q \leq 0$ we have 
\[
(\mathrm{Serre}(X))_{q} \simeq \mathrm{cofib}(X_{q+1} \rightarrow X_{-\infty}),
\]
where $X_{1}=0$ by convention. In particular, this implies that the Serre functor of $\presheaves(\mathbf{Z}_{\leq 0}; \ccat) \simeq \Fun(\mathbf{N}, \ccat)$ is an equivalence. 
\end{remark}

\subsection{Right Day convolution}

In this subsection, we record the observation that if $\ccat$ is a symmetric monoidal $\infty$-category with finite limits whose tensor product preserves finite limits in each variable separately, then $\Fun(\mathbf{N}, \ccat)$ admits a \emph{right} Day convolution symmetric monoidal structure. The latter is defined using limits in a way dual to the much more common left Day convolution. 

For example, this observation applies to any $\ccat$ which is stable and presentably symmetric monoidal. In this case, $\Fun(\mathbf{N}, \ccat)$ admits both the left and right Day convolution, and we show that the two are equivalent to each other with the equivalence provided by the Serre functor of $\Fun(\mathbf{N}, \ccat)$. 

\begin{recollection}[Left Day convolution]
Let $\ccat, \dcat$ be symmetric monoidal $\infty$-categories, with $\ccat$ small and $\dcat$ cocomplete. Given $X, Y \in \Fun(\ccat, \dcat)$, we can define the \emph{left} Day convolution by the formula 
\begin{equation}
\label{equation:left_day_convolution_formula}
(X \otimes_{\Fun(\ccat, \dcat)}^{L} Y)(c) \colonequals \varinjlim _{c_{1} \otimes_{\ccat} c_{2} \rightarrow c} X(c_{1}) \otimes Y(c_{2}),
\end{equation}
where the colimit is taken over $(c_{1}, c_{2}) \in (\ccat \times \ccat)_{/c}$. In other words, the left Day convolution is given by the left Kan extension 
\[
\begin{tikzcd}
	{\ccat \times \ccat} & {\dcat \times \dcat} & \dcat \\
	C & {}
	\arrow["{X \times Y}", from=1-1, to=1-2]
	\arrow["{\otimes_{\ccat}}", from=1-1, to=2-1]
	\arrow["{\otimes_{\dcat}}", from=1-2, to=1-3]
	\arrow[dotted, from=2-1, to=1-3]
\end{tikzcd}.
\]
\end{recollection}
As shown by Glasman \cite{glasman2017day}, if $\otimes_{\dcat} \colon \dcat \times \dcat \rightarrow \dcat$ preserves small colimits in each variable separately, then left Day convolution extends naturally to a symmetric monoidal structure on $\Fun(\ccat, \dcat)$. The monoidal unit $\mathbf{1}_{\Fun(\ccat, \dcat)}^{L}$ is given by the formula 
\begin{equation}
\label{equation:unit_of_day_convolution}
\mathbf{1}_{\Fun(\ccat, \dcat)}^{L}(c) \colonequals \varinjlim_{\mathbf{1}_{\ccat} \rightarrow c} \mathbf{1}_{\dcat}, 
\end{equation}
where the colimit is taken over the constant diagram indexed by $(\ccat^{\times 0})_{/c} \simeq \Map_{\ccat}(\mathbf{1}_{\ccat}, c)$. In \cite{winges2026hinich}, Winges shows that the assumptions on the existence and preservation of colimits can be severely weakened. The result is proven for general $\infty$-operads, but in the symmetric monoidal case relevant to us the only colimits that have to exist are those that appear in the formulas (\ref{equation:left_day_convolution_formula}) and (\ref{equation:unit_of_day_convolution}). 

\begin{theorem}[Winges]
\label{theorem:existence_of_left_day_convolution}
Let $\ccat, \dcat$ be symmetric monoidal $\infty$-categories. Suppose that for each $c \in \ccat$, $\dcat$ admits colimits indexed by
\begin{enumerate}
\item $\Map_{\ccat}(\mathbf{1}_{\ccat}, c)$, 
\item $(\ccat \times \ccat)_{/c}$
\end{enumerate}
and assume furthermore that such colimits are preserved by $\otimes_{\dcat}$ in each variable separately. Then (\ref{equation:left_day_convolution_formula}) extends to a symmetric monoidal structure on $\Fun(\ccat, \dcat)$. 
\end{theorem}

\begin{proof}
In \cite[{Proposition 4.1}]{winges2026hinich}, Winges shows that the result holds under the slightly stronger assumption that $\dcat$ admits and $\otimes_{\dcat}$ preserves colimits indexed by $(\ccat^{\times n})_{/c}$ for each $n \geq 0$ and each $c \in \ccat$. We claim that our two assumptions already imply the stronger one. There are three cases to consider:
\begin{enumerate}
    \item $n = 0$, the indexing $\infty$-category is given by $(\ccat^{\times 0})_{/c} \simeq \Map_{\ccat}(\mathbf{1}_{\ccat}, c)$, so the colimit exists and is preserved by assumption, 
    \item $n = 1$, the indexing $\infty$-category is given by $\ccat_{/c}$, which has a final object given by $\mathrm{id}_{c}$ and hence this colimit exists and is preserved with no assumptions on $\dcat$, 
    \item $n \geq 2$, the relevant colimit can be identified with the composite of left Kan extensions along 
    \[
    (\ccat^{\times n})_{/c} \rightarrow (\ccat^{\times n-1})_{/c} \rightarrow \ldots \rightarrow (\ccat \times \ccat)_{/c} \rightarrow \mathrm{pt},
    \]
where the first $(n-2)$ functors multiply the first two terms. As the last Kan extension exists and is preserved by assumption, it is enough to deal with the others. For each of the functors
\[
(\ccat^{\times k+1})_{/c} \rightarrow (\ccat^{\times k})_{/c}  
\]
for $k \geq 2$, the value of the relevant left Kan extension at $(b_{1}, \ldots, b_{k})$ is given by the colimit over the $\infty$-category
\[
(\ccat \times \ccat)_{/b_{1}} \times \ccat_{/b_{2}} \times \ccat_{/b_{3}} \times \ldots \times \ccat_{/b_{k}}. 
\]
The inclusion of $(\ccat \times \ccat)_{/b_{1}} \times \{ \mathrm{id}_{b_{2}} \} \times \ldots \{ \mathrm{id}_{b_{k}} \} \simeq (\ccat \times \ccat)_{/b_{1}}$ is a cofinal functor, hence this colimit is again preserved by assumption. 
\end{enumerate}
\end{proof}

By passing to opposites to interchange colimits and limits, we obtain the following dual construction:

\begin{variant}
\label{variant:right_day_convolution}
Let $\ccat, \dcat$ be symmetric monoidal $\infty$-categories and suppose that $\dcat$ admits \emph{limits} indexed by $\Map_{\ccat}(c, \mathbf{1}_{\ccat})$ and $(\ccat \times \ccat)_{c / }$ and that such limits are preserved by $\otimes_{\dcat}$ separately in each variable. Then $\Fun(\ccat^{\op}, \dcat^{\op})$ admits a symmetric monoidal structure given by left Day convolution. Under the canonical equivalence of $\infty$-categories 
\[
\Fun(\ccat, \dcat) \simeq \Fun(\ccat^{\op}, \dcat^{\op})^{\op}.
\]
provided by passing to opposites, this induces a symmetric monoidal structure on $\Fun(\ccat, \dcat)$ which we refer to as \emph{right Day convolution}. The underlying tensor product is given by 
\begin{equation}
\label{equation:formula_for_right_day_convolution}
(X \otimes_{\Fun(\ccat, \dcat)}^{R} Y)(c) \simeq \varprojlim_{c \rightarrow c_{1} \otimes_{\ccat} c_{2}} X(c_{1}) \otimes_{\dcat} Y(c_{2}), 
\end{equation}
where the limit is taken over $(c_{1}, c_{2}) \in (\ccat \times \ccat)_{c / }$. 
\end{variant}

We note that while any presentably symmetric monoidal $\infty$-category has a tensor product which preserves all colimits, it is a much more stringent condition for it to preserve limits, so the existence of the right Day convolution symmetric monoidal structure is much less commonplace. In the current work, we are interested in the following positive example: 

\begin{proposition}
\label{proposition:existence_of_right_day_for_fun_n_d_with_d_finitely_complete}
Let $\dcat$ be a finitely complete symmetric monoidal $\infty$-category such that $\otimes_{\dcat}$ preserves finite limits in each variable separately. Then $\Fun(\mathbf{N}, \dcat)$ admits a right Day convolution symmetric monoidal structure. 
\end{proposition}

\begin{proof}
By \cref{theorem:existence_of_left_day_convolution} (in \cref{variant:right_day_convolution}) it is enough to verify that for each $n \in \mathbf{N}$, $\dcat$ admits limits indexed by $\Map_{\mathbf{N}}(n, 0)$ and $(\mathbf{N} \times \mathbf{N})_{n / }$, and that such limits are preserved by $\otimes_{\dcat}$. The first $\infty$-category is either terminal or empty, and hence finite, so there is nothing to show. In the second case, we will show that there exists a limit-final functor 
\[
A_{n} \rightarrow (\mathbf{N} \times \mathbf{N})_{n / }
\]
with $A_{n}$ a finite poset. We can identify the target with the poset of pairs $(a, b)$ such that $a+b \geq n$, and we let $A_{n}$ be the subposet of those pairs which in addition satisfy $a+b \leq n+1$, pictured below 
\[
\begin{tikzcd}
	{(0, n+1)} &&& \\
	{(0, n)} & {(1, n)} \\
	& {(1, n-1)} & \ldots \\
	&& {(n, 0)} & {(n+1, 0)}
	\arrow[from=2-1, to=1-1]
	\arrow[from=2-1, to=2-2]
	\arrow[from=3-2, to=2-2]
	\arrow[from=3-2, to=3-3]
	\arrow[from=4-3, to=3-3]
	\arrow[from=4-3, to=4-4]
\end{tikzcd}
\]
Note that the geometric realization of $A_{n}$ is given by a "chain" $\Delta^{1} \cup_{\Delta^{0}} \Delta^{1} \cup_{\Delta^{0}} \ldots \cup_{\Delta^{0}} \Delta^{1}$ of intervals, which is contractible. By Quillen's Theorem A, it is enough to verify that for any $(a_{0}, b_{0})$ with $a_{0} + b_{0} \geq n$, the overcategory $(A_{n})_{/ (a_{0}, b_{0})}$ is weakly contractible. This corresponds to looking at the subcategory of elements below the vertical line given by $a_{0}$ and below the horizontal line given by $b_{0}$. This changes the geometric realization by removing an initial segment and a (disjoint) final segment of the chain, and therefore this overcategory is also weakly contractible as needed.  
\end{proof}

\begin{example}
The assumptions of \cref{proposition:existence_of_right_day_for_fun_n_d_with_d_finitely_complete} are satisfied, for example, when $\dcat$ is a stably symmetric monoidal $\infty$-category, or an abelian category with a tensor product exact in each variable. 
\end{example}

\begin{remark}
\label{remark:for_stable_presentable_d_right_day_on_fun_n_d_is_presentable}
If $\dcat$ is presentably symmetric monoidal stable $\infty$-category, then right Day convolution on $\Fun(\mathbf{N}, \dcat)$ is also presentably symmetric monoidal; that is, the tensor product preserves colimits in each variable separately. To see this, we note that the proof of 
\cref{proposition:existence_of_right_day_for_fun_n_d_with_d_finitely_complete} shows that the limit involved in the formula (\ref{equation:formula_for_right_day_convolution}) can be calculated over a finite $\infty$-category, and in a stable $\infty$-category finite limits commute with arbitrary colimits. 
\end{remark}

If $\ccat$ is a presentably symmetric monoidal stable $\infty$-category, then $\Fun(\mathbf{N}, \ccat)$ is a dualizable $\ccat$-module and so admits a $\ccat$-linear Serre functor $\Fun(\mathbf{N}, \ccat) \rightarrow \Fun(\mathbf{N}, \ccat)$ which we calculated in \cref{remark:serre_functor_of_non_positively_filtered_objects} to be given by the formula 
\[
(\mathrm{Serre}(X))_{q} \simeq \mathrm{cofib}(X_{q-1} \rightarrow X_{\infty}),
\]
where $X_{-1} = 0$ by convention and $X_{\infty} = \varinjlim X$\footnote{Note the change in indexing, as in \cref{remark:serre_functor_of_non_positively_filtered_objects} we worked in presheaves over $\mathbf{Z}_{\leq 0}$, but here we work with covariant functors on $\mathbf{N}$.}. This functor cannot be made symmetric monoidal with respect to either the left or right Day convolution alone. Instead, it intertwines the two in the following sense: 

\begin{proposition}
\label{proposition:serre_functor_on_fun_n_c_interchanges_left_and_right_day_convolution}
The Serre functor $S \colon \Fun(\mathbf{N}, \ccat) \rightarrow \Fun(\mathbf{N}, \ccat)$ is symmetric monoidal with respect to the left Day convolution on the source and the right Day convolution on the target. 
\end{proposition}

\begin{proof}
Since $S$ is a morphism of $\ccat$-modules in presentable $\infty$-categories, and the target is presentably symmetric monoidal by \cref{remark:for_stable_presentable_d_right_day_on_fun_n_d_is_presentable}, by the universal property of left Day convolution it is enough to give a symmetric monoidal structure on the composite $f$ in the diagram
\[
\begin{tikzcd}
	{\mathbf{N}^{\op}} & {\Fun(\mathbf{N}, \ccat)} & {\Fun(\mathbf{N}, \ccat)}
	\arrow["{y_{\ccat}}", from=1-1, to=1-2]
	\arrow["f"', bend right=30, from=1-1, to=1-3]
	\arrow["S", from=1-2, to=1-3]
\end{tikzcd},
\]
where $y_{\ccat}$ is the $\ccat$-linear Yoneda embedding. Using (\ref{equation:formula_for_serre_of_v_presheaves}), we see that this composite is given by 
\[
f(n)_{m} \simeq \underline{\map}_{\Fun(\mathbf{N},\ccat)}(y_{\ccat}(m), S(y_{\ccat}(n))) \simeq (\mathbf{1}_{\ccat}[\Map_{\mathbf{N}^{\op}}(n, m)])^{\vee}
\]
where $\mathbf{1}_{\ccat}[\Map_{\mathbf{N}^{\op}}(n, m)]$ is the colimit of the constant diagram on $\mathbf{1}_{\ccat}$ indexed by the space $\Map_{\mathbf{N}^{\op}}(n, m)$. As the latter is either empty or contractible, the image of $f$ is contained in the monoidal subcategory $\Fun(\mathbf{N}, \ccat_{\mathrm{dual}})$ of diagrams which are levelwise dualizable. 

We have a chain of symmetric monoidal equivalences 
\[
\Fun(\mathbf{N}, \ccat_{\mathrm{dual}})  \simeq \Fun(\mathbf{N}, \ccat_{\mathrm{dual}}^\op) \simeq \Fun(\mathbf{N}^{\op}, \ccat_{\mathrm{dual}})^{\op}, 
\]
where we view the left and middle $\infty$-categories as monoidal using right Day convolution, and the right one using the left Day convolution. Here, the first equivalence is given by passing to duals levelwise, and the second one is given by passing to opposite $\infty$-categories. Under these equivalences, $f$ corresponds to the opposite of the functor $g \colon \mathbf{N} \rightarrow \Fun(\mathbf{N}^{\op}, \ccat_{\mathrm{dual}})$ given by $g(n)_{m} \simeq \mathbf{1}_{\ccat}[\Map_{\mathbf{N}}(n, m)]$. This is the $\ccat$-linear Yoneda embedding, and hence is symmetric monoidal by the universal property of left Day convolution. 
\end{proof}

\section{Monoidality of the Tate filtration} 
\label{section:tate_filtration}

In this section we prove the $\mathbf{E}_{2}$-monoidality of the cellular Tate filtration. 

\begin{notation}
\label{notation:equivariant_spheres_giving_the_tate_filtration}
Let $\mathbf{C}$ be the topological space of complex numbers together with its canonical $U(1) \simeq S^{1}$-action. We write 
\[
\thesphere^{\mathbf{C}} \colonequals \Sigma^{\infty} \mathrm{N}(\mathbf{C} \cup \{ \infty \})
\]
for the suspension spectrum of (the singular nerve of) the one-point compactification of $\mathbf{C}$. This is a spectrum with an $S^{1}$-action which is non-equivariantly equivalent to the 2-sphere $\Sigma^{2} \thesphere$; in particular, it is invertible under the tensor product. For any $q \in \mathbf{Z}$, we write 
\[
\thesphere^{(q) \mathbf{C}} \colonequals (\thesphere^{\mathbf{C}})^{\otimes q}. 
\]
The element $0 \in \mathbf{C}$ is a fixed point under the $S^{1}$-action and hence it induces a map
\[
\alpha_{0} \colon \thesphere \rightarrow \thesphere^{\mathbf{C}},
\]
where we implicitly use the equivalence $\Sigma^{\infty} \mathrm{N}(\{ 0 \} \cup \{ \infty \}) \simeq \Sigma^{\infty}_{+} \mathrm{N}(\{ 0 \}) \simeq \thesphere$. Tensoring $\alpha_{0}$ with $\thesphere^{(q)\mathbf{C}}$ for varying $q$ yields a filtered object 
\[
\ldots \rightarrow \thesphere^{(q-1)\mathbf{C}} \rightarrow \thesphere^{(q)\mathbf{C}} \rightarrow \thesphere^{(q+1)\mathbf{C}} \rightarrow \ldots. 
\] 
which we will denote by $S^{\ast} \colonequals \thesphere^{-(\ast)\mathbf{C}} \in \Fil(\spectra^{h \mathrm{S^{1}}})$. 
\end{notation}

\begin{definition}
\label{definition:cellular_tate_filtration}
Let $X$ be a spectrum with an $S^{1}$-action. The associated \emph{cellular Tate filtration} is the filtered spectrum given by the formula
\[
q \in \ZZ \mapsto F^{q}_{T}(X^{tS^{1}}) \colonequals (X \otimes S^{q})^{hS^{1}} \simeq (X \otimes \thesphere^{-(q)\mathbf{C}})^{hS^{1}},  
\]
the homotopy fixed points of the tensor product of $X$ and the filtered spectrum of \cref{notation:equivariant_spheres_giving_the_tate_filtration}. 
\end{definition}

\begin{recollection}
\label{recollection:properties_of_the_bt_tate_filtration}
The basic properties of the cellular Tate filtration are as follows: 

\begin{enumerate}
    \item The colimit can be identified
    \[
    \varinjlim F^{\ast}_{T}(X^{tS^{1}}) \simeq X^{tS^{1}} \colonequals \mathrm{cofib}(X_{hS^{1}}[1] \rightarrow X^{hS^{1}})
    \]
    with the Tate construction \cite[{Proposition B.6}]{BalsPeriodicCyclicHomology}. Moreover, by \cite[{Theorem I.4.1}]{nikolaus-scholze} there is a unique such equivalence which commutes with the canonical maps from $F^{0}_{T}(X^{tS^{1}}) \simeq X^{hS^{1}}$. 
    \item The associated graded can be identified 
    \[
    \gr^{q}_{T}(X^{tS^{1}}) \simeq X[-2q]
    \]
    with the collection of even (de)suspensions of $X$. 
\end{enumerate}
\end{recollection}

\begin{theorem}
\label{theorem:lax_monoidal_structure_on_the_tate_filtration}
The Tate filtration of \cref{definition:cellular_tate_filtration} can be refined to a lax $\mathbf{E}_{2}$-monoidal functor 
\[
F_{T}^{\ast} \colon \spectra^{hS^{1}} \rightarrow \Fil \spectra.  
\]
\end{theorem}

We record that the above result is optimal in the following sense: 
\begin{warning}
\label{warning:tate_filtration_cannot_be_made_e3_monoidal}
The cellular Tate filtration cannot be promoted to an $\mathbf{E}_{3}$-monoidal functor, cf. \cite[{Warning 6.1.8}]{bhatt-lurie:apc}. To see this, take $X = \mathbf{F}_{2}$, the field with two elements. In this case, the non-negative part of the Tate filtration 
\[
\ldots \rightarrow (\mathbf{F}_{2} \otimes \thesphere^{(-2)\mathbf{C}})^{hS^{1}} \rightarrow (\mathbf{F}_{2} \otimes \thesphere^{-\mathbf{C}})^{hS^{1}}  \rightarrow (\mathbf{F}_{2} \otimes \thesphere)^{hS^{1}} 
\]
provides a filtration on $\mathbf{F}_{2}^{hS^{1}}$ which by a direct calculation is of the form 
\[
\ldots \rightarrow \tau_{\leq -4}(\mathbf{F}_{2}^{hS^{1}}) \rightarrow \tau_{\leq -2}(\mathbf{F}_{2}^{hS^{1}}) \rightarrow \mathbf{F}_{2}^{hS^{1}} 
\]
If we apply the doubling functor of \cref{recollection:evenification_functor_on_filtered_spectra} we obtain a filtered $\mathbf{F}_{2}$-module 
\[
\ldots \rightarrow \tau_{\leq -4}(\mathbf{F}_{2}^{hS^{1}}) \rightarrow \tau_{\leq -3}(\mathbf{F}_{2}^{hS^{1}}) \rightarrow \tau_{\leq -2}(\mathbf{F}_{2}^{hS^{1}}) \rightarrow \tau_{\leq -1}(\mathbf{F}_{2}^{hS^{1}}) \rightarrow \mathbf{F}_{2}^{hS^{1}},
\]
where we use that $\tau_{\leq -2n}(\mathbf{F}_{2}^{hS^{1}}) \simeq \tau_{\leq -2n+1}(\mathbf{F}_{2}^{hS^{1}})$ for all $n \geq 0$. This is in the heart of the Beilinson t-structure, and hence can be identified with a cochain complex of $\mathbf{F}_{2}$-vector spaces; in fact, the cellular cochain complex $\mathrm{C}^{\ast}(\mathbf{CP}^{\infty}, \mathbf{F}_{2})$ associated  to the standard cell structure on $\mathrm{B}S^{1} \simeq \mathbf{CP}^{\infty}$. In other words, the cellular Tate filtration provides an identification 
\begin{equation}
\label{equation:equivalence_between_f2hs1_and_a_realization_the_cellular_cochain_complex}
\mathbf{F}_{2}^{hS^{1}} \simeq | \mathrm{C}^{\ast}(\mathbf{CP}^{\infty}, \mathbf{F}_{2}) | \simeq | \Hrm^{\ast}(\mathbf{CP}^{\infty}, \mathbf{F}_{2}) |,
\end{equation}
where $| - |$ denotes the realization of a cochain complex in the sense of \cite[{Construction 3.28}]{antieau:decalage}; that is, the colimit of the associated filtered object. Here, the second equivalence follows from the observation that this cochain complex has zero differential, as $\mathbf{CP}^{\infty}$ has only even cells. 

If it were possible to promote the cellular Tate filtration to a lax $\mathbf{E}_{3}$-monoidal functor, then (\ref{equation:equivalence_between_f2hs1_and_a_realization_the_cellular_cochain_complex}) would be an equivalence of $\mathbf{E}_{3}$-algebras. In other words, $\mathbf{F}_{2}^{hS^{1}}$ would be formal as an $\mathbf{E}_{3}$-$\mathbf{F}_{2}$-algebra. This in turn would imply that the Dyer-Lashof operation $Q_{2}$ satisfies 
\[
Q_{2}(t) = 0, 
\]
where $t \in \pi_{-2}(\mathbf{F}_{2}^{hS^{1}})$ is the polynomial generator of $\pi_{-\ast}(\mathbf{F}_{2}^{hS^{1}}) \simeq \mathrm{H}^{\ast}(\mathbf{CP}^{\infty}, \mathbf{F}_{2}) \simeq \mathbf{F}_{2} [\![t]\!]$. However, on elements of cohomological degree two, $Q_{2}$ can be identified with $Q^{0}$, which acts by identity on $\mathbf{F}_{2}$-cochains of any space, see \cite[{Example 5.9}]{lawson2020n}, yielding a contradiction. Thus, the cellular Tate filtration cannot be made $\mathbf{E}_{3}$-monoidal. 

Note that this argument does not obstruct the Tate filtration from being lax $\mathbf{E}_{2}$-monoidal, as $\mathbf{F}_{2}^{hS^{1}}$ \emph{is} formal as an $\mathbf{E}_{2}$-$\mathbf{F}_{2}$-algebra; see  \cite[Lemma 2.1.10]{devalapurkar2024ku} or \cite[{Theorem 3.5}]{horel2025e_2}. 
\end{warning}

Note that, by construction, the cellular Tate filtration is a composite of two functors: 
\begin{enumerate}
    \item tensoring with the filtered spectrum with $S^{1}$-action $S^{\ast} \simeq \thesphere^{-(\ast)\mathbf{C}}$, 
    \item taking homotopy fixed points. 
\end{enumerate}

The latter functor is lax symmetric monoidal, so to prove \cref{theorem:lax_monoidal_structure_on_the_tate_filtration} it is enough to make the former lax $\mathbf{E}_{2}$-monoidal. This is equivalent to the following: 

\begin{theorem}
\label{theorem:suspension_spectra_of_complex_numbers_admits_a_filtered_e2_algebra_structure}
The filtered spectrum $S^{\ast}$ of \cref{notation:equivariant_spheres_giving_the_tate_filtration} admits an $\mathbf{E}_{2}$-algebra structure compatible with the $S^{1}$-action. 
\end{theorem}

We will prove the above result using an argument of Lurie, following the outline given in \S\ref{subsection:introduction_sketch_of_the_monoidality_argument}. The starting point is an $\mathbf{E}_{2}$-coalgebra structure on the filtered space $\{ \mathbf{CP}^{n} \}$ constructed in \cite[{Theorem 5.2.3}]{lurie2015rotation}. 

Each of the spaces $\mathbf{CP}^{n}$ maps into the colimit $\mathbf{CP}^{\infty} \simeq \mathrm{B}S^{1}$ and hence presents a space with an $S^{1}$-action. It will be important for us to keep track of this, hence we first need to verify that the coalgebra structure is compatible with the map into $\mathbf{CP}^{\infty}$. This is a consequence of the following general result: 

\begin{lemma}
\label{lemma:coalgebras_in_filtered_spaces_lift_to_coalgebras_in_spaces_over_their_colimit} 
Let $X = \{ X_{n} \}$ be a non-negatively, increasingly filtered space equipped with an $\mathbf{E}_{k}$-coalgebra structure with respect to (left) Day convolution on $\Fun(\mathbf{N}, \spaces)$. Then $\{ X_{n} \}$ can be lifted to an $\mathbf{E}_{k}$-coalgebra object of $\Fun(\mathbf{N}, \spaces_{/X_{\infty}})$, where $X_{\infty} \colonequals \varinjlim X_{n}$ is the colimit and we view $\spaces_{/X_{\infty}}$ as a symmetric monoidal $\infty$-category using the cartesian symmetric monoidal structure. 
\end{lemma}

\begin{proof}
The colimit functor admits a right adjoint $\const \colon \spaces \rightarrow \Fun(\mathbf{N}, \spaces)$ which is thus canonically oplax monoidal, and in this case in fact strongly monoidal. It follows that the unit map of this adjunction induces a morphism $X \rightarrow \const(X_{\infty})$ of $\mathbf{E}_{k}$-coalgebras. The target is in the image of a symmetric monoidal functor out of a cartesian symmetric monoidal $\infty$-category $\spaces$, and thus the given $\mathbf{E}_{k}$-coalgebra structure extends uniquely to the unique cocommutative coalgebra structure by the dual of \cite[{Proposition 2.4.3.9}]{HA}. 

The coalgebra $\const(X_{\infty})$ determines a commutative algebra object in the symmetric monoidal $\infty$-category $\Fun(\mathbf{N}, \spaces)^{\op}$ which we will denote by $A$. By \cite[{Theorem 3.3.3.9}]{HA}, the undercategory $\ccat \colonequals (\Fun(\mathbf{N}, \spaces)^{\op})_{A /}$ can be extended to an $\infty$-operad 
\[
\ccat^{\otimes} \colonequals \Mod^{\mathbf{E}_{\infty}}_{A}(\Fun(\mathbf{N}, \spaces)^{\op})^{\otimes}
\]
with the property that for a collection of objects
\[
(A \rightarrow B^{i}), (A \rightarrow B) \in (\Fun(\mathbf{N}, \spaces)^{\op})_{A /}
\]
the corresponding space 
\[
\mathrm{Mul}_{\ccat} \Big(\{ (A \rightarrow B^{1}), \ldots, (A \rightarrow B^{n}) \}, (A \rightarrow B)\Big) 
\]
of multimorphisms can be identified with the space of maps $B^{1} \otimes \ldots \otimes B^{n} \rightarrow B$ which make the diagram
\[
\begin{tikzcd}
	{A \otimes \ldots \otimes A} & A \\
	{B^{1} \otimes \ldots \otimes B^{n} } & B
	\arrow[from=1-1, to=1-2]
	\arrow[from=1-1, to=2-1]
	\arrow[from=1-2, to=2-2]
	\arrow[from=2-1, to=2-2],
\end{tikzcd}\]
commute. Here, the unmarked tensor product is that of $\Fun(\mathbf{N}, \spaces)^{\op}$ and the upper horizontal arrow is multiplication of $A$. 

In our case of working with the opposite $\infty$-category, we can think of $B^{i}$ and $B$ as filtered spaces \emph{over} $\const(X_{\infty})$ and the above multispace is the space of maps $B \rightarrow B^{1} \otimes_{\mathrm{Day}} \ldots \otimes_{\mathrm{Day}} B^{n}$ of filtered spaces together with the datum of a commutative diagram 
\[
\begin{tikzcd}
	B & {B^{1} \otimes_{\mathrm{Day}} \ldots \otimes_{\mathrm{Day}} B^{n}} \\
	{\const(X_{\infty})} & {\const(X_{\infty}) \otimes_{\mathrm{Day}} \ldots \otimes_{\mathrm{Day}}\const(X_{\infty})}
	\arrow[from=1-1, to=1-2]
	\arrow[from=1-1, to=2-1]
	\arrow[from=1-2, to=2-2]
	\arrow[from=2-1, to=2-2]
\end{tikzcd}
\]
As the bottom map is induced by the unique comultiplication of a coalgebra in spaces, it can be identified with the cartesian diagonal $\Delta \colon  \const(X_{\infty}) \rightarrow \const(X_{\infty}) \times \ldots \times \const(X_{\infty})$. If we fix $B^{1}, \ldots, B^{n}$ and consider the above space as a functor of $B$, it is representable by the pullback in filtered spaces 
\[
(B^{1} \otimes_{\mathrm{Day}} \ldots \otimes_{\mathrm{Day}} B^{n}) \times_{\const(X_{\infty}) \times \ldots \times \const(X_{\infty})} \const(X_{\infty}). 
\]
Explicitly, in filtered degree $k \in \mathbf{N}$, this is given by 
\[
(\varinjlim B^{1}_{a_{1}} \times \ldots \times B^{n}_{a_{n}}) \times_{X_{\infty} \times \ldots \times X_{\infty}} X_{\infty},
\] 
where the colimit is taken over the poset of tuples of natural numbers such that $a_{1}+\ldots+a_{n} \leq k$. As pullbacks in $\spaces$ are universal, we can rewrite this as 
\[
\varinjlim \Big((B^{1}_{a_{1}} \times \ldots \times B^{n}_{a_{n}}) \times_{X_{\infty} \times \ldots \times X_{\infty}} X_{\infty}\Big) \simeq \varinjlim \Big( B^{1}_{a_{1}} \times_{X_{\infty}} \ldots \times_{X_{\infty}} B^{n}_{a_{n}} \Big). 
\] 
The right-hand side here is the formula for left Day convolution of $\Fun(\mathbf{N}, \spaces_{/X_{\infty}})$. Putting this all together, we have shown that 
\[
\mathrm{Mul}_{\ccat} (\{ B^{1}, \ldots, B^{n} \}, B ) \simeq \Map_{\Fun(\mathbf{N}, \spaces_{/X_{\infty}})}(B, B^{1} \otimes_{\mathrm{Day}} \ldots \otimes_{\mathrm{Day}} B^{n}).
\]
In other words, the $\infty$-operad $\ccat^{\otimes}$ is in fact a symmetric monoidal $\infty$-category, equivalent to $\Fun(\mathbf{N}, \spaces_{/X_{\infty}})^{\op}$.

By \cite[{Corollary 3.4.1.7}]{HA}, an $\mathbf{E}_{k}$-algebra object in $\ccat^{\otimes}$ is the same as a morphism $A \rightarrow B$ of $\mathbf{E}_{k}$-algebras in $\Fun(\mathbf{N}, \spaces)^{\op}$. Dualizing, this means that an $\mathbf{E}_{k}$-coalgebra object in 
\[
\ccat^{\op} \simeq \Fun(\mathbf{N}, \spaces_{/X_{\infty}}) 
\]
is equivalent to an $\mathbf{E}_{k}$-coalgebra in filtered spaces together with a morphism of $\mathbf{E}_{k}$-coalgebras into $\const(X_{\infty})$. Applying this to the unit map $X \rightarrow \const(X_{\infty})$ shows that $X$ lifts to a coalgebra in filtered spaces over $X_{\infty}$, as needed. 
\end{proof}

\begin{proof}[{Proof of \cref{theorem:suspension_spectra_of_complex_numbers_admits_a_filtered_e2_algebra_structure}:}]
By \cite[{Theorem 5.2.3}]{lurie2015rotation}, the filtered space 
\[
\mathbf{CP}^{0} \rightarrow \mathbf{CP}^{1} \rightarrow \mathbf{CP}^{2} \rightarrow \ldots 
\]
admits an $\mathbf{E}_{2}$-coalgebra structure as an object of $\Fun(\mathbf{N}, \spaces)$ equipped with left Day convolution. By \cref{lemma:coalgebras_in_filtered_spaces_lift_to_coalgebras_in_spaces_over_their_colimit}, this filtered space can be lifted to an $\mathbf{E}_{2}$-coalgebra in the $\infty$-category 
\[
\Fun(\mathbf{N}, \spaces_{/\mathbf{CP}^{\infty}}) \simeq \Fun(\mathbf{N}, \spaces^{h S^{1}})
\]
of filtered spaces over $\varinjlim \mathbf{CP}^{n} \simeq \mathbf{CP}^{\infty} \simeq \mathrm{B}S^{1}$, which we can identify with filtered spaces with $S^{1}$-action.  

Passing to suspension spectra, we obtain an $\mathbf{E}_{2}$-coalgebra in filtered spectra with an $S^{1}$-action of the form 
\[
\Sigma^{\infty}_{+} S(\mathbf{C}) \rightarrow \Sigma^{\infty}_{+} S(\mathbf{C}^{2}) \rightarrow \Sigma^{\infty}_{+} S(\mathbf{C}^{3}) \rightarrow \ldots,   
\]
where $S(\mathbf{C}^{k})$ is the $(2k-1)$-dimensional sphere of unit vectors in $\mathbf{C}^{k}$. That is, $\Sigma^{\infty}_{+} S(\mathbf{C}^{\ast+1})$ is a coalgebra in $\Fun(\mathbf{N}, \spectra^{hS^{1}})$, where we consider the latter as symmetric monoidal with respect to (left) Day convolution. The $\infty$-category $\Fun(\mathbf{N}, \spectra^{hS^{1}})$ is dualizable over $\spectra^{hS^{1}}$ and so admits a Serre functor 
\[
\mathrm{Serre} \colon \Fun(\mathbf{N}, \spectra^{hS^{1}}) \rightarrow \Fun(\mathbf{N}, \spectra^{hS^{1}})
\]
which as we observed in \cref{remark:serre_functor_of_non_positively_filtered_objects} is given by the formula\footnote{Note the change in indexing, as in \cref{remark:serre_functor_of_non_positively_filtered_objects} we work with presheaves over $\mathbf{Z}_{\leq 0}$, rather than covariant functors indexed by $\mathbf{N}$.} 
\[
\mathrm{Serre}(X)_{q} \simeq \mathrm{cofib} (X_{q-1} \rightarrow X_{\infty}), 
\]
where $X_{\infty} \colonequals \varinjlim X_{\ast}$ and $X_{-1} = 0$. As we show in \cref{proposition:serre_functor_on_fun_n_c_interchanges_left_and_right_day_convolution}, the Serre functor is symmetric monoidal with respect to left Day convolution on the source and right Day convolution on the target. Applying it to the coalgebra of the previous paragraph, we see that $B \colonequals \mathrm{Serre}(\Sigma^{\infty}_{+} S(\mathbf{C}^{\ast+1}))$, of the form  
\[
B_{q} \simeq \mathrm{cofib}(\Sigma^{\infty}_{+}S(\mathbf{C}^{q}) \rightarrow \varinjlim \Sigma^{\infty}_{+}S(\mathbf{C}^{\infty})) \simeq \thesphere^{(q)\mathbf{C}},
\]
admits a structure of an $\mathbf{E}_{2}$-coalgebra in $\Fun(\mathbf{N}, \spectra^{hS^{1}})$ with respect to right Day convolution. This filtered object is levelwise dualizable, and passing to monoidal duals gives an object of $\presheaves(\mathbf{N}; \spectra^{hS^{1}})$ of the form 
\[
\ldots \rightarrow \thesphere^{(-2)\mathbf{C}} \rightarrow \thesphere^{(-1)\mathbf{C}} \rightarrow \thesphere.
\]
As taking monoidal duals takes finite limits to colimits, this is an $\mathbf{E}_{2}$-algebra in the usual sense, i.e. with respect to left Day convolution. By left Kan extending along the inclusion $\mathbf{N} \hookrightarrow \mathbf{Z}$, we obtain an $\mathbf{E}_{2}$-algebra object $R$ of $\presheaves(\mathbf{Z}; \spectra^{hS^{1}}) = \Fil^{\down}(\spectra^{hS^{1}})$ of the form 
\[
\ldots \rightarrow \thesphere^{(-2)\mathbf{C}} \rightarrow \thesphere^{(-1)\mathbf{C}} \rightarrow \thesphere \rightarrow \thesphere \rightarrow \thesphere \rightarrow \ldots. 
\]
The inclusion of the first filtered piece determines a map of filtered spectra with $S^{1}$-action 
\[
\alpha \colon (1)_{!}(\thesphere^{(-1)\mathbf{C}}) \rightarrow R, 
\]
where the source is the filtered object freely generated by $\thesphere^{(-1)\mathbf{C}}$ in degree $1$. The latter is invertible under the tensor product, so using the methods of \cite[{\S 7.2.3}]{HA} we can form the localization $R[\alpha^{-1}]$ which since $R$ is $\mathbf{E}_{2}$ and hence satisfies the Ore condition can be identified with the colimit of the diagram 
\[
R \rightarrow R \otimes (-1)_{!}(\thesphere^{(1)\mathbf{C}}) \rightarrow R \otimes (-2)_{!}(\thesphere^{(2)\mathbf{C}}) \rightarrow \ldots.
\]
After unwrapping what this means, we see that $R[\alpha^{-1}]$ is an $\mathbf{E}_{2}$-algebra of the form 
\[
\ldots \rightarrow \thesphere^{(-2)\mathbf{C}} \rightarrow \thesphere^{(-1)\mathbf{C}} \rightarrow \thesphere \rightarrow \thesphere^{(1)\mathbf{C}} \rightarrow \thesphere^{(2)\mathbf{C}} \rightarrow \ldots
\]
as required. 
\end{proof}

\section{Doubling functor on filtered spectra} 

In this section, we study the ``doubling'' functor $d_{!} \colon \Fil\spectra \rightarrow \Fil\spectra$. We make explicit how the functor changes the associated spectral sequence, see \cref{remark:spectral_sequence_of_evenification}, and how it interacts with d\'{e}calage, see \cref{proposition:decalage_and_doubling_interaction}. 

\begin{recollection}
\label{recollection:evenification_functor_on_filtered_spectra}
We will write $d \colon \ZZ \rightarrow \ZZ$ for the multiplication by two functor. This induces an adjunction 
\[
d_{!} \dashv d^{*} \colon \Fil \spectra \rightleftarrows \Fil \spectra 
\]
where $d^{*}$ is given by precomposition and $d_{!}$ is the left Kan extension. Explicitly, the value of $d_{!}$ at $X \in \Fil \spectra$ is given by the formula
\[
F^{q}(d_{!} X) \simeq F^{\ceil{\frac{q}{2}}} X. 
\]
At the level of associated graded objects, we instead have 
\[
\gr^{q}(d_{!} X) \simeq  
\begin{cases}
\gr^{\frac{q}{2}}(X) \textnormal{ if } q  \textnormal{ is even} \\
0 \textnormal{ otherwise} 
\end{cases}.
\]
\end{recollection}

\begin{recollection}[{\cite[Lemma 5.14]{van2025introduction}}]
Since $d$ is fully faithful, so is $d_{!}$, and so its essential image is a coreflective subcategory. Moreover, for a filtered spectrum $X$ the following three conditions are equivalent: 
\begin{enumerate}
    \item $F^{q} X \rightarrow F^{q-1} X$ is an equivalence for each $q \in 2 \ZZ$, 
    \item $\gr^{q} X = 0$ for each $q \in 2 \ZZ + 1$, 
    \item $X$ is in the essential image of $d_{!}$. 
\end{enumerate} 
\end{recollection}

\begin{remark}
\label{remark:spectral_sequence_of_evenification}
A filtered spectrum $X$ has an associated spectral sequence of signature
\[
E^{1}_{s, t}(X) \simeq \pi_{s+t}(\gr^{-s} X) \Rightarrow \pi_{t+s}(| X |)
\]
where $| X | \colonequals \varinjlim X$ is the colimit. Using the explicit formulas of \cite[{Construction 1.2.2.6}]{HA}, it is not difficult to see that the spectral sequence associated to the evenification $d_{!}X$ can be identified, up to regrading, with the spectral sequence of $X$. Note that Lurie's construction agrees with the one based on d\'{e}calage by \cite[{Remark 4.15}]{antieau:decalage}. 

Explicitly, the formulas are as follows. For all $s, t \in \ZZ$ and all $r \geq 1$ we have 
\begin{enumerate}
    \item $E^{2r-1}_{2s, t}(d_{!}X) \simeq E^{2r}_{2s, t}(d_{!} X) \simeq E^{r}_{s, s+t}(X)$, 
    \item $E^{r}_{2s-1, t}(d_{!}X) = 0$, 
\end{enumerate}
Under the first isomorphism, the odd page differentials vanish while 
\[
d_{2r}^{d_{!}X} \colon E^{2r}_{2s, t}(d_{!}X) \rightarrow E^{2r}_{2s-2r, t+2r}(d_{!}X)
\]
can be identified with 
\[
d_{r}^{X} \colon E^{r}_{s, s+t}(X) \rightarrow E^{r}_{s-r, s+t+r}(X). 
\]
\end{remark}

\begin{lemma}
\label{lemma:homotopy_of_decalage_of_evenification}
Let $X$ be a filtered spectrum. Then for every $a, k \in \ZZ$ we have an isomorphism 
\[
(\pi^{B}_{2a} (\Dec(d_{!}X)))^{k} \simeq (\pi^{B}_{a} X)^{k-a}
\]
of chain complexes. In odd degrees, we instead have $\pi^{B}_{2a-1}(\Dec(d_{!}X)) = 0$. 
\end{lemma}

\begin{proof}
Since $(\pi_{s}^{B} X)^{t} \simeq E^{1}_{-t, s}(X)$, the first statement is equivalent to
\[
E^{1}_{-k, 2a}(\Dec(d_{!}X)) \simeq E^{2}_{2a-2k, k}(d_{!}X) \simeq E^{1}_{2a-2k, k}(d_{!}X) \simeq E^{1}_{a-k, a}(X)
\]
where the last two equivalences are \cref{remark:spectral_sequence_of_evenification}. The second follows in the same way. 
\end{proof}

\begin{lemma}
\label{lemma:behavior_of_d\'{e}calage_of_doubling_in_beilinson_t_structure}
For each $a \in \ZZ$, the map $\Dec(d_{!}(\tau_{\geq a}^{B} X)) \rightarrow \Dec(d_{!}X)$ is a Beilinson $2a$-connective cover. 
\end{lemma}

\begin{proof}
If $X \simeq \widehat{X}$ is complete, the statement follows from \cref{lemma:homotopy_of_decalage_of_evenification}. In general, taking the fibre $F \colonequals \mathrm{fib}(X \rightarrow \widehat{X})$, we have a commutative diagram 
\[
\begin{tikzcd}
	{\Dec(d_{!}\tau_{\geq a}^{B}F)} & {\Dec(d_{!}\tau_{\geq a}^{B}X)} & {\Dec(d_{!}\tau_{\geq a}^{B} \widehat{X})} \\
	{\Dec(d_{!}F)} & {\Dec(d_{!}X)} & {\Dec(d_{!}\widehat{X})}
	\arrow[from=1-1, to=1-2]
	\arrow[from=1-1, to=2-1]
	\arrow[from=1-2, to=1-3]
	\arrow[from=1-2, to=2-2]
	\arrow[from=1-3, to=2-3]
	\arrow[from=2-1, to=2-2]
	\arrow[from=2-2, to=2-3]
\end{tikzcd}
\]
where rows are fibre sequences. The left vertical map is an equivalence of constant filtered spectra and the result follows. 
\end{proof}

\begin{proposition}
\label{proposition:decalage_and_doubling_interaction}
For any $X \in \Fil(\spectra)$ there is a natural equivalence of filtered spectra 
\[
\Dec(\Dec(d_{!}X)) \simeq d_{!} \Dec(X). 
\]
\end{proposition}

\begin{proof}
By \cref{lemma:behavior_of_d\'{e}calage_of_doubling_in_beilinson_t_structure}, the left-hand side is an even filtered spectrum, so it is enough to verify the equivalence in even degrees. If $a \in \ZZ$, then 
\[
\Dec_{2a}(\Dec(d_{!}X)) \simeq | \tau_{\geq 2a}^{B} \Dec(d_{!} X) | \simeq | \Dec(d_{!} \tau_{\geq a}^{B} X) | \simeq | \tau_{\geq a}^{B} X | \simeq \Dec_{a}(X),
\]
where $|-|$ denotes the colimit of a filtered spectrum. Here we use that neither the d\'{e}calage nor $d_{!}$ change the colimit. 
\end{proof}

\begin{remark}
\label{remark:bifiltered_refinement_of_the_interaction_between_decalage_and_doubling}
The equivalence of \cref{proposition:decalage_and_doubling_interaction} can be lifted to an equivalence of \emph{bifiltered} objects using functoriality. More precisely, if we write $F^{s} X$ for the filtered spectrum defined by 
\[
(F^{s} X)_{t} \colonequals \begin{cases} X_{t} \textnormal{ for } t \geq s \\
X_{s} \textnormal{ otherwise}  
\end{cases}, 
\]
then
\[
\Dec_{2a}(\Dec(d_{!} F^{s}X)) \simeq \Dec_{a}(F^{s}X) \simeq (\pi_{\geq a}^{B} X)_{s}.
\]
\end{remark}

\section{Comparison with the homotopy Tate filtration}

In this section, we compare the cellular and homotopy Tate filtrations, and record explicitly the isomorphism between the associated spectral sequences. 

\begin{theorem}
\label{theorem:decalage_of_evenification_of_bhatt_lurie_tate_is_the_tate_of_postnikov}
Let $X$ be a spectrum with $S^{1}$-action and let $F^{\ast}_{T}(X^{tS^{1}})$ be the cellular Tate filtration of \cref{definition:cellular_tate_filtration}. Then 
\begin{enumerate}
    \item for every $a \in \ZZ$, the map 
    \[
    d_{!} F_{T}^{\ast}((\tau_{\geq a} X)^{tS^{1}}) \rightarrow d_{!} F_{T}^{\ast}(X^{tS^{1}})
    \]
    is an $a$-connective cover with respect to the Beilinson t-structure, 
    \item there is an equivalence of filtered spectra
    \[
    \Dec(d_{!}(F_{T}^{\ast}(X^{tS^{1}}))) \simeq (\tau_{\geq \ast} X)^{tS^{1}}.
    \]
\end{enumerate}
\end{theorem}

\begin{proof}
We follow Antieau's argument of \cite[Proposition 9.2]{antieau:decalage}. By a combination of \cref{recollection:properties_of_the_bt_tate_filtration} and \cref{recollection:evenification_functor_on_filtered_spectra}, 
\[
\gr^{q}(d_{!} F_{T}^{\ast}((\tau_{\geq a} X)^{tS^{1}})) \simeq  
\begin{cases}
(\tau_{\geq a} X)[-q] \textnormal{ if } q \textnormal{ is even} \\
0 \textnormal{ otherwise},
\end{cases}
\]
so that the source of the map in $(1)$ is Beilinson $a$-connective. Since the cofibre is given by $d_{!} F_{T}^{\ast}((\tau_{< a}X)^{tS^{1}})$, the same formula shows that the cofibre is $(a-1)$-coconnective, proving the first part. The second part follows from the first by taking colimits along the cellular Tate filtration and \cref{recollection:properties_of_the_bt_tate_filtration}.
\end{proof}

\begin{corollary}
\label{corollary:bt_tate_and_standard_tate_have_isomorphic_spectral_sequences}
Up to suitable reindexing, the spectral sequences associated to $F^{\ast}_{T}(X^{tS^{1}})$ and $(\tau_{\geq \ast} X)^{tS^{1}}$ are isomorphic. More precisely, for each $r \geq 1$ we have 
\[
E^{r}_{s, s+t}(F_{T}^{\ast}(X^{tS^{1}})) \simeq E^{2r}_{2s, t}(d_{!}(F_{T}^{\ast}(X^{tS^{1}}))) \simeq E^{2r-1}_{-t, 2s+2t}((\tau_{\geq \ast} X)^{tS^{1}}). 
\]
\end{corollary}

\begin{proof}
This is a combination of \cref{remark:spectral_sequence_of_evenification} and Antieau's comparison \cite[{Theorem 4.13}]{antieau:decalage}. 
\end{proof}

\begin{corollary}
\label{corollary:decalaged_cellular_tate_is_double_of_decalaged_postnikov_tate}
If $X$ is a spectrum with an $S^{1}$-action, there is a natural equivalence of filtered spectra 
\[
d_{!} \Dec(F_{T}^{\ast}(X^{tS^{1}})) \simeq \Dec((\tau_{\geq \ast} X)^{tS^{1}}). 
\]
In other words, the single-speed d\'{e}calage of the cellular Tate filtration is equivalent to the double-speed d\'{e}calage of the homotopy Tate filtration. 
\end{corollary}

\begin{proof}
This is a combination of \cref{theorem:decalage_of_evenification_of_bhatt_lurie_tate_is_the_tate_of_postnikov} and \cref{proposition:decalage_and_doubling_interaction}. 
\end{proof}

\section{The synthetic Tate filtration} 
\label{section:synthetic_tate_filtration}

In this section, we discuss a synthetic variant of the cellular Tate filtration which previously appeared in the work of Antieau-Riggenbach \cite[{Proof of Lemma 2.75}]{antieau2024cyclotomic}, and we calculate its colimit and associated graded. The input for this construction is a synthetic spectrum equipped with a compatible action of the circle; that is, a module over the Antieau-Riggenbach-Raksit even circle $\evencircle$. 

\begin{notation} The \emph{synthetic sphere} is the filtered ring spectrum 
\[
\synsphere \colonequals \fil_{\pev}^{\ast}(\thesphere) \in \CAlg(\Fil \spectra), 
\]
the perfect even filtration of the sphere. We refer to $\synsphere$-modules as \emph{synthetic spectra}\footnote{In the terminology of \cite{pstrkagowski2023synthetic}, $\synspectra$ denotes \emph{even, $\MU$-based} synthetic spectra.} and denote their $\infty$-category by 
\[
\synspectra \colonequals \Mod_{\synsphere}(\Fil \spectra). 
\]
The \emph{even circle} is the filtered ring spectrum given by 
\[
\evencircle \colonequals \fil^{\ast}_{\pev}(\thesphere[S^{1}]) \in \CAlg(\synspectra),
\]
the perfect even filtration of the suspension spectrum of the circle. 
\end{notation}

\begin{remark}[{Even or perfect even?}]
Since we will apply the even filtration functor to modules, it will be convenient to work with the perfect even filtration of \cite{pstrkagowski2025perfect}, rather than the Hahn-Raksit-Wilson even filtration of \cite{hahnmotivic}. For the two algebras considered above, $\thesphere$ and $\thesphere[S^{1}]$, the two coincide by \cite[{Theorem 7.5}]{pstrkagowski2025perfect}, so that our notation above agrees with \cite{antieau2024cyclotomic}. 
\end{remark}

\begin{example}
\label{example:hrk_filtered_hochschild_homology_as_a_module_over_the_even_circle}
By \cite[{Corollary 2.44}]{antieau2024cyclotomic}, we have
\[
\fil_{\ev}^{\ast}(\ZZ) \otimes_{\synsphere} \evencircle \simeq  \tau_{\geq 2 \ast}(\ZZ) \otimes_{\synsphere} \evencircle \simeq \tau_{\geq \ast} (\ZZ[S^{1}]), 
\]
so that $\evencircle$ is a lift of Raksit's $\ZZ$-linear filtered circle $\tau_{\geq \ast}(\ZZ[S^{1}])$ to the synthetic sphere. In particular, if $A \rightarrow B$ is a morphism of derived rings, then Raksit's $\mathrm{HKR}$-filtered Hochschild homology $\HH_{\fil}(B/A)$ of \cite[{Definition 6.2.1}]{raksit:2020} is a $\evencircle$-module by restriction of scalars. 
\end{example}
\begin{notation}[{Filtration shift}]
If $X$ is a filtered spectrum, we write $X(k)$ for the filtered spectrum obtained by filtration shift; that is, for each $q \in \ZZ$ we have 
\[
(X(k))_{q} \colonequals X_{q-k}. 
\]
This should not be confused with the categorical suspension, which we denote by $X[1] \colonequals \Sigma X$. 
\end{notation}

\begin{recollection}[{Bialgebra structure and self-duality}]
By \cite[{Corollary 2.44}]{antieau2024cyclotomic}, the bialgebra structure on $\thesphere[S^{1}]$ induces a bialgebra structure on $\evencircle$ in $\synsphere$-modules. In particular, we have an augmentation map $\evencircle \rightarrow \synsphere$ which makes the synthetic sphere into a module over the even circle. 
\end{recollection}

\begin{recollection}[{Self-duality}]
\label{recollection:dual_of_the_even_circle}
The bialgebra $\evencircle$ is self-dual up to a shift. More precisely, by \cite[{Lemma 2.47}]{antieau2024cyclotomic}, we have an equivalence of $\evencircle$-modules 
\[
\evencircle^{\vee} \simeq \evencircle[-1](-1),
\]
where $\evencircle^{\vee} \colonequals \underline{\map}_{\synsphere}(\evencircle, \synsphere)$ is the monoidal dual. 
\end{recollection}

\begin{recollection}[Filtered orbits, fixed points and the Tate construction]
An $\evencircle$-module $X$ can be thought of as a filtered spectrum equipped with an action of the circle compatible with the filtration. One defines the filtered fixed points and orbits by the formulas 
\[
X_{\evencircle} \colonequals \synsphere \otimes_{\evencircle} X, \qquad
X^{\evencircle} \colonequals \underline{\map}_{\evencircle}(\synsphere, X),
\]
where on the right-hand side we have the internal $\synsphere$-module mapping object. By \cite[{Construction 2.58}]{antieau2024cyclotomic}, there is a norm map 
\[
\mathrm{Nm}_{X} \colon (X_{\evencircle})(1)[1] \rightarrow X^{\evencircle} 
\]
whose cofibre we denote by $X^{t \evencircle} \colonequals \mathrm{cofib}(\mathrm{Nm}_{X})$ and call the \emph{filtered Tate construction}. Note that the shift appearing in the source comes from the identification of \cref{recollection:dual_of_the_even_circle}. 
\end{recollection}

We now explain how to equip the filtered Tate construction with an additional auxiliary filtration, a filtered variant of the cellular Tate filtration. 

\begin{recollection}
\label{recollection:equivariant_perfect_even_filtration}
By \cite[{Theorem 3.37}]{pstrkagowski2025perfect}, the module perfect even filtration can be assembled into a functor 
\[
\fil_{\pev, \thesphere[S^{1}]}(-) \colon \Mod_{\thesphere[S^{1}]}(\spectra) \rightarrow \Mod_{\evencircle}(\synspectra). 
\]
We explain how its values can be calculated in practice. By a combination of \cite[{Corollary 2.36}]{antieau2024cyclotomic} and \cite[{Proposition 2.2.20}]{hahnmotivic}, the composite 
\[
\thesphere[S^{1}] \rightarrow \thesphere \rightarrow \MU
\]
of the augmentation and the unit map of complex bordism is faithfully even flat. It follows from \cite[{Theorem 6.26}]{pstrkagowski2025perfect} that for any $X \in \spectra^{\mathrm{B}S^{1}}$ there is a map of filtered spectra
\[
\fil_{\pev, \thesphere[S^{1}]}(X) \rightarrow \mathrm{Tot}(\tau_{\geq 2 \ast}(\MU^{\otimes_{\thesphere[S^{1}]} \bullet} \otimes_{\thesphere[S^{1}]} X)) 
\]
which identifies the target with the completion of the source. If $X$ is bounded below, which is the only case of interest for us, \cite[Theorem 8.3]{pstrkagowski2025perfect} shows that the perfect even filtration is already complete, and thus this map is an equivalence. The $\evencircle$-module structure comes from the identification 
\[
\evencircle \simeq \fil^{\ast}_{\pev}(\thesphere[S^{1}]) \simeq \mathrm{Tot}(\tau_{\geq 2 \ast}(\MU^{\otimes_{\thesphere[S^{1}]} \bullet}))
\]
\end{recollection}

\begin{notation}
\label{notation:filtered_synthetic_spectrum_inducing_the_equivariant_even_filtration}
We write 
\[
S^{\ast}_{\ev} \colonequals \fil_{\pev, \thesphere[S^{1}]}(S^{\ast}) \in \Fil(\Mod_{\evencircle}(\synspectra))
\]
for the image of the filtered $\thesphere[S^{1}]$-module of \cref{notation:equivariant_spheres_giving_the_tate_filtration} under the functor of \cref{recollection:equivariant_perfect_even_filtration}.
\end{notation}

\begin{definition}
\label{definition:tate_filtration_in_synthetic_spectra}
Let $X$ be an $\evencircle$-module. The associated \emph{synthetic Tate filtration} is the filtered synthetic spectrum
\[
F^{\ast}_{T}(X^{t \evencircle}) \colonequals (S^{\ast}_{\ev} \otimes_{\synsphere} X)^{h \evencircle}. 
\]
\end{definition}

\begin{warning}
Beware that our notation is potentially abusive: $F_{T}^{\ast}(X^{t \evencircle})$ is really a functor of $X$ rather than $X^{t \evencircle}$. We chose this convention to stay compatible with \cite[{\S 6.1}]{bhatt-lurie:apc}. 
\end{warning}

Note that, unlike for the ordinary variant, we do not show in this paper that the formation of the synthetic Tate filtration can be made $\mathbf{E}_{2}$-lax monoidal. Instead, we make the following:

\begin{conjecture}
\label{conjecture:monoidality_of_the_synthetic_tate_filtration}
We conjecture that: 
\begin{enumerate}
    \item the functor 
    \[
    \fil_{\pev, \thesphere[S^{1}]}(-) \colon \Mod_{\thesphere[S^{1}]}(\spectra) \rightarrow \Mod_{\evencircle}(\synspectra), 
    \]
    can be made lax symmetric monoidal, when we consider both sides with the symmetric monoidal structures coming from the bialgebra structure on, respectively, $\thesphere[S^{1}]$ and $\evencircle$, 
    \item the synthetic Tate filtration of \cref{definition:tate_filtration_in_synthetic_spectra} can be promoted to a lax $\mathbf{E}_{2}$-monoidal functor
    \[
    \Mod_{\evencircle}(\synspectra) \rightarrow \Fil(\synspectra).
    \] 
\end{enumerate}
\end{conjecture}

\begin{remark}
Note that the second part of \cref{conjecture:monoidality_of_the_synthetic_tate_filtration} should follow from the first, since Raksit's filtered fixed points are lax symmetric monoidal. 
\end{remark}

In the rest of this section, we calculate the associated graded and the colimit of the synthetic Tate filtration. 

\begin{proposition}
\label{proposition:associated_graded_of_the_synthetic_tate_filtration}
Let $X$ be a $\evencircle$-module. Then the associated graded of the synthetic Tate filtration is given by 
\[
\gr^{n}_{T}(X^{t \evencircle}) \simeq X[-2n](-n). 
\]
\end{proposition}

\begin{proof}
We claim that for each $q \in \ZZ$, the fibre sequence
\[
\Sigma^{-2q-2} \thesphere[S^{1}] \rightarrow \thesphere^{-(q+1)\mathbf{C}} \rightarrow \thesphere^{-(q)\mathbf{C}}
\]
is preserved by the functor $\fil^{\ast}_{\pev, \thesphere[S^{1}]}(-)$ of \cref{recollection:equivariant_perfect_even_filtration}. As $\thesphere[S^{1}] \rightarrow \MU$ is faithfully even flat by \cite[{Corollary 2.36}]{antieau2024cyclotomic}, it is enough to check that the base-change 
\[
\MU \otimes_{\thesphere[S^{1}]} \Sigma^{-2q-2} \thesphere[S^{1}] \rightarrow \MU \otimes_{\thesphere[S^{1}]} \thesphere^{-(q+1)\mathbf{C}} \rightarrow \MU \otimes_{\thesphere[S^{1}]} \thesphere^{-(q)\mathbf{C}}
\]
is short exact on homotopy groups. On homotopy, the first map can be identified with the morphism
\[
\MU_{*}(\Sigma^{-2q-2} \thesphere) \rightarrow \MU_{*}(\Sigma^{-2q-2} \mathbf{CP}^{\infty}) 
\]
induced by the inclusion of a bottom cell, which is injective as needed. We deduce that the associated graded of the filtered synthetic spectrum of \cref{notation:filtered_synthetic_spectrum_inducing_the_equivariant_even_filtration} can be calculated as 
\[
\gr^{q}(S_{\ev}) \simeq (\fil^{\ast}_{\pev, \thesphere[S^{1}]}(\Sigma^{-2q-2} \thesphere [S^{1}]))[1] \simeq (\evencircle[-2q-2](-q-1))[1] \simeq \evencircle^{\vee}[-2q](-q),
\]
where the last identification uses \cref{recollection:dual_of_the_even_circle}. The statement about the associated graded of the synthetic Tate filtration follows, since 
\[
\gr^{q}_{T}(X^{t \evencircle}) \simeq (\gr^{q}(S_{\ev}) \otimes_{\synsphere} X)^{\evencircle} \simeq (\evencircle^{\vee}[-2q](-q) \otimes_{\synsphere} X)^{\evencircle} \simeq X[-2q](-q). 
\]
\end{proof}

\begin{lemma}
\label{lemma:colimit_of_the_synthetic_tate_filtration_annihilates_induced_modules}
Let $X \colonequals \evencircle \otimes_{\synsphere} Y$ be an induced $\evencircle$-module. Then 
\[
\varinjlim F^{q}_{T}(X^{t \evencircle}) = 0, 
\]
where the colimit is taken as $q \to -\infty$. 
\end{lemma}

\begin{proof}
By \cref{recollection:dual_of_the_even_circle}, up to a shift we can identify $X$ with $Y \otimes_{\synsphere} \evencircle^{\vee}$. We then have 
\[
\varinjlim F^{\ast}_{T}((Y \otimes_{\synsphere} \evencircle^{\vee})^{t \evencircle}) \simeq \varinjlim S^{\ast}_{\ev} \otimes_{\synsphere} Y.
\]
The right-hand side has a vanishing colimit since
\[
\varinjlim S^{\ast}_{\ev} \simeq \fil_{\pev, \thesphere[S^{1}]}(\thesphere^{-(\ast)\mathbf{C}}) \simeq \fil_{\pev, \thesphere[S^{1}]}(\thesphere^{(\infty)\mathbf{C}}) \simeq 0 
\]
and $\thesphere^{(\infty)\mathbf{C}} \simeq 0$, where we use \cite[{Lemma 2.24}]{pstrkagowski2025perfect} to commute the perfect even filtration past the filtered colimit. 
\end{proof}

\begin{proposition}
\label{proposition:colimit_of_the_synthetic_tate_filtration}
There is a unique natural transformation $\psi$ of functors $\Mod_{\evencircle}(\synspectra) \rightarrow \synspectra$ making the diagram 
\[\begin{tikzcd}
	{F^{0}_{T}(X^{t \evencircle}) \simeq X^{\evencircle}} && {\varinjlim F^{\ast}_{T}(X^{t \evencircle})} \\
	& {X^{t \evencircle}} & %\bullet
	\arrow[from=1-1, to=1-3]
	\arrow["{\mathrm{can}}", from=1-1, to=2-2]
	\arrow["\psi"', from=2-2, to=1-3]
\end{tikzcd}
\]
commute. Moreover, this natural transformation is an equivalence. 
\end{proposition}

\begin{proof}
The first part follows immediately from the universal property of the Tate construction of \cite[{Proposition 2.4.10}]{raksit:2020} and \cref{lemma:colimit_of_the_synthetic_tate_filtration_annihilates_induced_modules}. For the second part, we observe that by the proof of \cref{proposition:associated_graded_of_the_synthetic_tate_filtration} each of the maps 
\[
S^{q+1}_{\ev} \otimes_{\synsphere} X \rightarrow S^{q}_{\ev} \otimes_{\synsphere} X 
\]
becomes an equivalence after applying $(-)^{t \evencircle}$. Thus, we have a compatible system of identifications
\[
(S^{q}_{\ev} \otimes_{\synsphere} X)^{t \evencircle} \simeq (S^{0}_{\ev} \otimes_{\synsphere} X)^{t \evencircle} \simeq X^{t \evencircle}.
\]
It follows that the cofibre of $\psi$ can be identified with 
\[
\varinjlim \mathrm{cofib}(X^{t \evencircle} \rightarrow (S^{\ast}_{\ev} \otimes_{\synsphere} X)^{\evencircle}) \simeq \varinjlim \Sigma (S^{\ast}_{\ev} \otimes_{\synsphere} X)_{\evencircle} \simeq \Sigma((\varinjlim S^{\ast}_{\ev}) \otimes_{\synsphere} X)_{\evencircle}
\] 
which vanishes since $\varinjlim S^{\ast}_{\ev} = 0$. 
\end{proof}

\section{The three HKR filtrations on $\HC^{-}$ and $\HP$}\label{section:HKR_comparison}

In this section, we apply the synthetic Tate filtration of 
\cref{definition:tate_filtration_in_synthetic_spectra} and the equivalence of \cref{theorem:decalage_of_evenification_of_bhatt_lurie_tate_is_the_tate_of_postnikov} to compare three variants of the $\HKR$-filtration on periodic cyclic and negative cyclic homology appearing in the literature. The three filtrations in question are due to Antieau, Raksit and Bhatt--Lurie. We first recall their constructions. 

\begin{notation}
In this section, $k$ denotes a classical commutative ring and $\dcat(k) \simeq \Mod_{k}(\spectra)$ its derived $\infty$-category. 
\end{notation}

\begin{recollection}
To an animated $k$-algebra $A$ one can associate its Hochschild homology spectrum $\HH(A / k) \in \dcat(k)$, as well as negative cyclic and periodic homology
\[
\HC^{-}(A / k) \colonequals \HH(A / k)^{hS^{1}}, \textnormal{ } \HP(A / k) \colonequals \HH(A / k)^{tS^{1}}. 
\]
Hochschild homology comes equipped with a Hochschild--Kostant--Rosenberg filtration, and this allows one to construct filtered spectra 
\[
F^{\ast}_{\HKR} \HC^{-}(A / k), \textnormal{ } F^{\ast}_{\HKR} \HP(A / k),
\]
with associated graded given by 
\[
\gr_{\HKR}^{\ast} \HC^{-}(A / k) \simeq (\mathrm{dR}_{A/k}^{\wedge, \geq \ast})[2 \ast], \textnormal{ } \gr_{\HKR}^{\ast} \HP(A / k) \simeq (\mathrm{dR}_{A/k}^{\wedge})[2 \ast], 
\]
shifts of (truncated, in the case of $\HC^{-}$) Hodge-completed derived de Rham cohomology. 
\end{recollection}

At least three constructions of the above $\HKR$ filtration appear in the literature.

\begin{recollection}[{Antieau, \cite{antieau2019periodic}}]
\label{recollection:antieaus_hkr}
To a finitely generated polynomial $k$-algebra $P \in \Poly_{k}$ we associate the bifiltered object\footnote{In \cite{antieau2019periodic}, the cellular Tate filtration is denoted by $F^{\ast}_{\mathrm{CW}}$ rather than the $F^{\ast}_{T}$ used here.}
\[
F^{r}_{H} F^{s}_{T} \HP(P/k) \colonequals F^{s}_{T}((\tau_{\geq r} \THH(P/k))^{tS^{1}}) 
\]
given by the cellular Tate filtration applied to the homotopy filtration of $\THH(P/k)$. This object is bicomplete, as are all of its Beilinson connective covers with respect to the homotopy filtration. We denote the \emph{double-speed} tower of Beilinson connective covers as 
\[
F^{\ast}_{BH} F^{\ast}_{H} F^{\ast}_{T} \HP(P/k) \colonequals \tau_{\geq 2\ast}^{B} (F^{\ast}_{H} F^{\ast}_{T} \HP(P/k))
\] 
Here, the subscript $BH$ is to remind the reader that the Beilinson tower is taken with respect to the homotopy filtration. The above association admits a unique extension to a sifted colimit-preserving functor
\[
F^{\ast}_{BH} F^{\ast}_{H} F^{\ast}_{T} \HP(-/k) \colon \CAlg_{k}^{\an} \rightarrow \Fun(\ZZ^{\op}, \compbiFil(\dcat(k)))
\] 
defined on all animated $k$-algebras. If $A$ is an animated $k$-algebra, Antieau's $\HKR$-filtration on the associated $\HC^{-}$ and $\HP$ is given by 
\begin{align*}
F^{\ast}_{\AntHKR} \HC^{-}(A/k) &\colonequals F^{\ast}_{BH} F^{0}_{H} \HC^-(A/k), \\
F^{\ast}_{\AntHKR} \HP(A/k) &\colonequals F^{\ast}_{BH} F^{0}_{H} \HP(A/k),
\end{align*}
where the second definition is shorthand for first taking the colimit along the Tate-filtration in $\Fun(\Zbf^\op, \Fil^c(\Dc(k)))$ and then passing to the $F^0_H$-part. Note that our definition of $F^{\ast}_{\AntHKR} \HP(A/k)$ looks different from Antieau's definition in \cite[Proof of Corollary~4.12]{antieau2019periodic}. The equivalence of these two definitions is explained in \cref{rem:ben_HP_HKR}.
\end{recollection}

\begin{remark}[Role of the Tate filtration in Antieau's construction]\label{rem:ben_tate}
Let us explain the role of the cellular Tate filtration in Antieau's construction. For a polynomial $k$-algebra $P$, we have 
\[
\gr^r_H \HC^-(P/k) = (\gr^r_H \HH(P/k))^{hS^{1}} \simeq (\Omega^r_{P/k}[r])^{hS^1} \simeq \bigoplus_{i=0}^\infty \Omega^r_{P/k}[r-2i]
\]
where the last equality uses the fact that $\Omega^r_{P/k}$ is a discrete module. This functor left Kan extends to $A \mapsto \bigoplus_{i=0}^\infty \bigwedge^r L_{A/k}[r-2i]$ which for a general animated $k$-algebra $A$ does not coincide with 
\[
\gr^r_H \HC^-(A/k) = (\gr^r_H \HH(A/k))^{hS^{1}} \simeq
    \Big(\bigwedge^r L_{A/k}[r]\Big)^{hS^1} \simeq \prod_{i=0}^\infty \Big(\bigwedge^r L_{A/k}[r-2i]\Big).
\]
In other words, $F_H^* \HC^-(A/k)$ is not left Kan extended from polynomial algebras as a complete single-filtered complex. The Tate filtration fixes this by breaking the infinite direct sum down into individual summands, essentially completing the infinite sum of cotangent complexes into an infinite product. For this reason, $F_H^* F_T^* \HC^-(A/k)$ \emph{is} left Kan extended from polynomial algebras as a bicomplete bifiltered complex, as observed in \cite[{Remark 4.4}]{antieau2019periodic}. 

We remark that this Tate completion is not necessary for $\HC(A/k) \colonequals \HH(A/k)_{hS^1}$, as in this case 
\[
\gr^r_H \HC(P/k) \simeq (\Omega^r_{P/k}[r])_{hS^1} \simeq \bigoplus_{i=0}^\infty \Omega^r_{P/k}[r+2i]
\]
\emph{does} left Kan extend to 
\[
A \mapsto \bigoplus_{i=0}^\infty \Big(\bigwedge^r L_{A/k} \Big)[r+2i] \simeq \Big(\bigwedge^r L_{A/k}[r]\Big)_{hS^1} \simeq \gr^r_H \HC(A/k).
\]
Thus, $F^*_H \HC(A/k) = F^*_H |F^*_T \HC(A/k)|$ is left Kan extended as a complete single-filtered complex from polynomial algebras. 
\end{remark}

\begin{remark}[On Antieau's $\HKR$-filtration on $\HP$]\label{rem:ben_HP_HKR}
Antieau defines his HKR-filtration on $\HP$ of polynomial $k$-algebras $P$ as the cofibre
\begin{equation*}
    F_{\AntHKR'}^* F_H^* \HP(P/k) := \mathrm{cofib}\left(\tau^B_{\geq 2\ast  - 1} (F_H^* \HC (P/k)[1]) \to \tau^B_{\geq 2\ast} (F_H^* \HC^-(P/k))\right),\footnote{Antieau suppresses the filtration-shift from his notation. One sees that our $\tau^B_{\geq 2\ast  - 1} (F_H^* \HC (P/k)[1])$ coincides with Antieau's Beilinson filtration on $\HC$ by comparing the graded pieces.}
\end{equation*}
and then left Kan extends it to all animated $k$-algebras. The resulting filtration coincides with our $F_{\AntHKR}^* \HP(A/k)$. First, we observe that
\[
F_H^* \HC^-(P/k) \to  F_H^* \HP(P/k) \to  F_H^* \HC(P/k)[2]
\]
is a cofibre sequence of filtered objects that induces split cofibre sequences on graded pieces. Thus, taking Beilinson connective covers sends it to a cofibre sequence. Combining this observation with the formula
\[
\tau^B_{\geq 2n} (F_H^* \HC(P/k)[2]) = (\tau^B_{\geq 2n-1}F_H^* \HC(P/k)[1])[1],
\]
we obtain a cofibre sequence
\[
\tau^B_{\geq 2\ast - 1} (F_H^* \HC(P/k)[1]) \to  \tau^B_{\geq 2\ast} (F_H^* \HC^-(P/k)) \to  \tau^B_{\geq 2\ast} (F_H^* \HP(P/k)).
\]
Passing to $F^0_H$, the rightmost entry in this cofibre sequence becomes our $F_{\AntHKR}^* \HP(P/k)$, proving the claim on polynomial $k$-algebras. As both filtrations are left Kan extended, they coincide on all animated $k$-algebras.
\end{remark}

\begin{recollection}[{Bhatt--Lurie, \cite{bhatt-lurie:apc}}]
\label{recollection:bhatt_lurie_hkr_filtration}
Given $P \in \Poly_{k}$ we write
\[
F^{\ast}_{BT} F^{\ast}_{T} \HP(P / k) \colonequals \tau^{B}_{\geq \ast}(F_{T}^{\ast}(\THH(P / k)^{tS^{1}})),  
\]
for the Beilinson--Whitehead tower of the cellular Tate filtration on $\THH(P / k)^{tS^{1}}$. Here, the subscript $BT$ is supposed to remind the reader of the fact that the Beilinson covers are taken with respect to the cellular Tate filtration. 

If $A$ is a general animated $k$-algebra, we define $ F^{\ast}_{BT} F^{\ast}_{T} \HP(A / k)$ as the value on $A$ of the unique extension of the above association to a sifted colimit-preserving functor 
\[
F^{\ast}_{BT} F^{\ast}_{T} \HP(- / k) \colon \CAlg^{\an}_{k} \rightarrow \Fun(\ZZ^{\op}, \compFil(\dcat(k)))
\]
valued in objects which are complete with respect to the Tate filtration.
\end{recollection}

\begin{warning}
Observe that both Antieau's and Bhatt--Lurie's definitions are quite similar, involving a left Kan extension of a suitable d\'{e}calage operation. Beware that these d\'{e}calage are, however, taken with respect to two different filtrations, namely the homotopy (i.e. Whitehead) filtration and the cellular Tate filtration! The comparison between the two we prove in this section says that these two a priori different operations yield the same result. 
\end{warning}

\begin{recollection}[{Raksit}]
\label{recollection:raksit_hkr_filtration}
 In \cite[{Definition 6.2.5}]{raksit:2020}, Raksit characterizes a certain module $\HH_{\mathrm{fil}}(A / k) \in \Mod_{\evencircle}\Fil(\dcat(k))$ over the filtered circle by a universal property. To keep our notation consistent we will denote this module by 
\[
F^{\ast}_{\RakHKR} \HH(A / k) \colonequals \HH_{\mathrm{fil}}(A / k) \in \Mod_{\evencircle}\Fil(\dcat(k)).
\]
Raksit's $\HKR$-filtration on $\HC^{-}$ and $\HP$ are defined as 
\begin{align*}
    F^{\ast}_{\RakHKR} \HC^{-}(A / k) &\colonequals (F^{\ast}_{\RakHKR} \HH(A / k))^{\evencircle}, \\
F^{\ast}_{\RakHKR} \HP(A / k) &\colonequals (F^{\ast}_{\RakHKR} \HH(A / k))^{t\evencircle}, 
\end{align*}
the filtered variant of fixed points and Tate construction of \cite[{\S 6.3}]{raksit:2020}.
\end{recollection}

The following is the main result of this section. 

\begin{theorem}
\label{theorem:antieau_bhattlurie_raksit_hkr_filtrations_on_hcminus_are_equivalent}
Let $A$ be an animated $k$-algebra. Then there are equivalences, natural in $A$,
\[
F^{\ast}_{\AntHKR} F_T^* \HP(A / k) \simeq F^{\ast}_{\BLHKR} F_T^*\HP(A / k) \simeq F^*_T F^{\ast}_{\RakHKR} \HP(A / k) \in \mathrm{BiFil}(\dcat(k)), 
\]
where the second equivalence exchanges the filtrations. In particular, the three HKR filtrations on $\HC^-$ and $\HP$ coincide.
\end{theorem}

Note that the comparison of filtrations on $\HC^{-}$ is obtained by restricting to non-negative Tate-filtration degrees since $F^{0}_{T} \HP(A/k) \simeq \HC^{-}(A/k)$ in each case. We prove the two equivalences of \cref{theorem:antieau_bhattlurie_raksit_hkr_filtrations_on_hcminus_are_equivalent} separately, dealing with the first one in the lemma below and with the second one in \cref{proposition:bhattlurie_raksit_hkr_filtrations_on_hcminus_are_equivalent}. 

\begin{lemma}
We have $F^{\ast}_{\AntHKR} F^*_T \HP(A / k) \simeq F^{\ast}_{\BLHKR} F^*_T \HP(A/k)$. In particular, the HKR-filtrations of Antieau and Bhatt--Lurie on $\HP$ and $\HC^-$ coincide. 
\end{lemma}

\begin{proof} 
First, suppose that $A \in \Poly_{k}$. In this case, \cref{theorem:decalage_of_evenification_of_bhatt_lurie_tate_is_the_tate_of_postnikov} gives an identification of bifiltered objects 
\[
d_{!} F^{\ast}_{H} F^{\ast}_{T} \HP(A/k) \simeq \Dec_{\ast}(d_{!} F_{T}^{\ast}(\HH(A/k)^{tS^{1}}))
\]
where the homotopy filtration on the left corresponds to the d\'{e}calage filtration on the right. Taking double-speed d\'{e}calage in this direction, we obtain an equivalence of trifiltered objects 
\[
F^{\ast}_{BH} F^{\ast}_{H} F^{\ast}_{T} \HP(A/k) \simeq \Dec_{2 \ast}(\Dec_{\ast}(d_{!} F^{\ast}_{T}(\HH(A/k)^{tS^{1}})))
\]
that proves the claim on polynomial algebras.

For general $A$, the left-hand side is defined as the unique sifted colimit-preserving extension when considered as a functor in the $\infty$-category of objects which are complete in the Tate and homotopy directions. By \cite[{Lemma~4.8}]{antieau:decalage}, this is the same as the unique extension when valued in tricomplete objects, which as a consequence of \cite[{Lemma~4.9}]{antieau:decalage} then coincides with the unique extension valued in objects which are complete in the Tate and Beilinson-homotopy directions. Working with that last extension and restricting to homotopy filtration zero, we deduce that for any animated $A$ we have 
\[
F^{\ast}_{BH} F^{0}_{H} F^{\ast}_{T} \HP(A/k) \simeq 
\varinjlim \Dec_{2 \ast}(\Dec_{0}(d_{!} F^{\ast}_{T}(\THH(B/k)^{tS^{1}}))), 
\]
where the colimit is taken over $B \in (\Poly_{k})_{/A}$ inside $\compbiFil(\dcat(k))$. Using the bifiltered refinement of \cref{proposition:decalage_and_doubling_interaction} given in \cref{remark:bifiltered_refinement_of_the_interaction_between_decalage_and_doubling}, 
we can rewrite the right-hand side as 
\[
\varinjlim \Dec_{2 \ast}(\Dec_{0}(d_{!} F^{\ast}_{T}(\THH(B/k)^{tS^{1}}))) \simeq \varinjlim \Dec_{\ast}(F^{\ast}_{T}(\THH(B/k)^{tS^{1}}))
\]
By \cite[{Remark 6.3.5}]{bhatt-lurie:apc}, this colimit gives the same result when calculated in objects which are complete only with respect to the Tate filtration. The latter colimit is exactly the one which appears in the construction of the Bhatt--Lurie filtration. Putting these together, we obtain the desired equivalence 
\[
F^{\ast}_{BH} F^{0}_{H} F^{\ast}_{T} \HP(A/k) \simeq F^{\ast}_{BT} F^{\ast}_{T} \HP(A/k), 
\]
finishing the proof.
%which when restricted to Tate degree zero yields $F^{\ast}_{\AntHKR} \HC^{-}(A/k) \simeq F^{\ast}_{\BLHKR} \HC^{-}(A/k)$, as needed. 
\end{proof}

We move on to the second equivalence. Since $F^{\ast}_{\BLHKR}$ is defined by left Kan extending a certain bifiltered object, to compare it to $F^{\ast}_{\RakHKR}$ we will equip the latter with an additional filtration, given by the synthetic Tate filtration of \S \ref{section:synthetic_tate_filtration}. 

\begin{notation}
\label{notation:tate_filtration_on_raksits_filtered_hp}
If $A$ is an animated ring, we write 
\begin{equation}
\label{equation:synthetic_tate_of_raksits_hkr_hh}
F_{T}^{\ast} F^{\ast}_{\RakHKR} \HP(A / k)
\end{equation}
for the synthetic Tate filtration of \cref{definition:tate_filtration_in_synthetic_spectra} applied to the $\evencircle$-module $F^{\ast}_{\RakHKR} \HH(A / k)$. Since
\[
| F^{\ast}_{\RakHKR} \HH(A/k) | \simeq \HH(A/k), 
\]
there is a canonical map
\[
| F^{\ast}_{T} F^{\ast}_{\RakHKR} \HP(A/k)| \rightarrow F_{T}^* \HP(A/k). 
\]
This map is an equivalence on the associated graded objects, and hence identifies the target with the completion of the source. 
\end{notation}

\begin{lemma}
\label{lemma:tate_filtered_raksits_hp_is_bicomplete}
For any animated $A$, $F^{\ast}_{T} F^{\ast}_{\RakHKR} \HP(A/k)$ is complete as a bifiltered object. 
\end{lemma}

\begin{proof}
For each $q \in \ZZ$, we have 
\[
F_{T}^{q} F^{\ast}_{\RakHKR} \HP(A / k) \simeq (\ZZ[-2q](-q)  \otimes_{\ZZ} F^{\ast}_{\RakHKR} \HH(A / k))^{h \evencircle}
\]
which is complete as a filtered object by a combination of \cite[{Proposition 6.2.9, Proposition 6.3.7}]{raksit:2020}. Since the synthetic Tate filtration is complete, the result follows. 
\end{proof}

\begin{lemma}
\label{lemma:synthetic_tate_filtered_raksits_hp_is_a_beilinson_tower_for_polynomial_rings}
Let $A \in \Poly_{k}$ be a finitely generated polynomial $k$-algebra. Then for every $s \in \ZZ$, the canonical map 
\[
F_{T}^{\ast} F_{\RakHKR}^{s} \HP(A / k) \rightarrow F_{T}^{\ast} \HP(A / k) 
\]
identifies the source with the Beilinson $s$-connective cover of the target. 
\end{lemma}

\begin{proof}
We have to show that for each $q \in \ZZ$, the map 
\[
\gr_{T}^{q} F^{s}_{\RakHKR} \HP(A / k) \rightarrow \gr_{T}^{q} \HP(A / k) 
\]
is an $(s-q)$-connective cover of spectra. Using the description of the associated graded of the synthetic Tate filtration given in \cref{proposition:associated_graded_of_the_synthetic_tate_filtration}, we can identify this map with 
\[
(F^{s+q}_{\RakHKR} \THH(A / k))[-2q] \rightarrow \THH(A / k)[-2q].
\]
By \cite[{Remark 6.2.10}]{raksit:2020}, we can further rewrite it as 
\[
(\tau_{\geq s+q} \THH(A / k))[-2q] \rightarrow \THH(A / k)[-2q]
\]
which ends the argument. 
\end{proof}

\begin{proposition}
\label{proposition:bhattlurie_raksit_hkr_filtrations_on_hcminus_are_equivalent}
We have a filtration-exchanging equivalence of bifiltered objects 
\[
F^*_T F^{\ast}_{\RakHKR} \HP(A/k) \simeq F^{\ast}_{\BLHKR} F^*_T \HP(A/k).
\]
In particular, the HKR-filtrations of Raksit and Bhatt--Lurie on $\HP$ and $\HC^-$ coincide. 
\end{proposition}

\begin{proof}
Consider the functor
\[
F^{\ast}_{T} F^{\ast}_{\RakHKR} \HP(- / k) \colon \CAlg_{k}^{\an} \rightarrow \compbiFil(\dcat(k)). 
\]
of \cref{notation:tate_filtration_on_raksits_filtered_hp}, which lands in complete bifiltered spectra by \cref{lemma:tate_filtered_raksits_hp_is_bicomplete}. As a consequence of \cref{lemma:synthetic_tate_filtered_raksits_hp_is_a_beilinson_tower_for_polynomial_rings}, when restricted to the category $\Poly_{k}$ of finitely generated polynomial $k$-algebras, this functor agrees with the d\'{e}calage of the cellular Tate filtration. The construction of $F^{\ast}_{\BLHKR}$ as a left Kan extension thus provides a canonical natural transformation of functors 
\begin{equation}
\label{equation:nat_of_functors_in_bifiltered_spectra_in_ant_rak_comparison}
F^{r}_{BT} F^{s}_{T} \HP(- / k) \rightarrow F^{s}_{T} F^{r}_{\RakHKR} \HP(- / k)
\end{equation}
valued in complete bifiltered spectra. Note that here our use of $r, s \in \ZZ$ is to remind the reader that the external/internal directions get interchanged. 

We claim that (\ref{equation:nat_of_functors_in_bifiltered_spectra_in_ant_rak_comparison}) is an equivalence, for which it is enough to check that the target preserves sifted colimits. This can be verified after passing to the associated graded, where it is a consequence of the formula 
\[
\gr^{s}_{T} \gr^{r}_{\RakHKR} \HP(A / k) \simeq \gr^{r+s}_{\RakHKR} \HH(A / k)[-2s] \simeq (\Lambda^{r+s} L_{A/k})[-2s]
\]
which combines \cref{proposition:associated_graded_of_the_synthetic_tate_filtration} and \cite[{Theorem 6.2.6}]{raksit:2020}. 
\end{proof}

\bibliographystyle{alphamod}
\bibliography{references}

\end{document}

%% file: preamble.tex
\usepackage[T1]{fontenc}
\usepackage[utf8]{inputenc}
\usepackage{amsmath, amsthm, amsfonts, amssymb}

\usepackage{euscript} % Lurie-type Euler script
\usepackage[usenames, dvipsnames]{xcolor} % colourfultext
\usepackage{url}
\usepackage{nicefrac}
\usepackage[all]{xy}
\usepackage{placeins} % barriers for floats
\usepackage{musicography} % allows for flats and sharps

\usepackage{bbm} % mathbb numbers
\usepackage{enumitem} % fancy enumerate

\usepackage{aliascnt}

\usepackage{hyperref}
\hypersetup{
    colorlinks, linkcolor=NavyBlue,
    citecolor=OliveGreen, urlcolor=RedOrange
}

\usepackage{faktor}

\usepackage{tikz-cd}
\usepackage{tikz}
\usetikzlibrary{arrows,backgrounds,
    decorations.pathreplacing,
    decorations.pathmorphing
}
\usetikzlibrary{matrix,arrows}
\usepackage{colonequals}
\usetikzlibrary{decorations.pathmorphing,fit}

\newcommand{\euscr}[1]{\EuScript{#1}} % Euler script
\newcommand{\ccat}{\euscr{C}} % category C in Euler script 
\newcommand{\dcat}{\euscr{D}} % category D in Euler script
\newcommand{\ecat}{\euscr{E}} % category E in Euler script 
\newcommand{\vcat}{\euscr{V}} % category P in Euler script
\newcommand{\map}{\textnormal{map}} % mapping space
\newcommand{\spaces}{\euscr{S}} % the category of spaces
\newcommand{\evencircle}{\mathbf{T}_{\mathrm{ev}}}

\newcommand{\synsphere}{\thesphere_{\mathrm{syn}}}
\newcommand{\synspectra}{\mathrm{Syn}}
\newcommand{\ZZ}{\mathbf{Z}}
\newcommand{\HKR}{\mathrm{HKR}}
\newcommand{\HP}{\mathrm{HP}}
\newcommand{\HC}{\mathrm{HC}} 
\newcommand{\an}{\mathrm{an}} 

\newcommand{\const}{\mathrm{const}}
\newcommand{\spectra}{\euscr{S}p} % the category of spectra
\newcommand{\thesphere}{\mathbf{S}} % the sphere spectrum
\def\MU{\mathrm{MU}}

\newcommand{\nocontentsline}[3]{}
\newcommand{\tocless}[2]{\bgroup\let\addcontentsline=\nocontentsline#1{#2}\egroup}

\theoremstyle{plain}

\newtheorem{theorem}{Theorem}[section]

\newaliascnt{lemma}{theorem}
\newtheorem{lemma}[lemma]{Lemma}
\aliascntresetthe{lemma}

\newaliascnt{proposition}{theorem}
\newtheorem{proposition}[proposition]{Proposition}
\aliascntresetthe{proposition}

\usepackage{mathtools}
\DeclarePairedDelimiter{\ceil}{\lceil}{\rceil}
\newaliascnt{corollary}{theorem}
\newtheorem{corollary}[corollary]{Corollary}
\aliascntresetthe{corollary}

\newaliascnt{conjecture}{theorem}
\newtheorem{conjecture}[conjecture]{Conjecture}
\aliascntresetthe{conjecture}

\newaliascnt{question}{theorem}

\aliascntresetthe{question}

\theoremstyle{definition}

\newaliascnt{example}{theorem}
\newtheorem{example}[example]{Example}
\aliascntresetthe{example}

\newaliascnt{variant}{theorem}
\newtheorem{variant}[variant]{Variant}
\aliascntresetthe{variant}

\newaliascnt{warning}{theorem}
\newtheorem{warning}[warning]{Warning}
\aliascntresetthe{warning}

\newaliascnt{recollection}{theorem}
\newtheorem{recollection}[recollection]{Recollection}
\aliascntresetthe{recollection}

\newaliascnt{movable}{theorem}

\aliascntresetthe{movable}

\newaliascnt{definition}{theorem}
\newtheorem{definition}[definition]{Definition}
\aliascntresetthe{definition}

\newaliascnt{remark}{theorem}
\newtheorem{remark}[remark]{Remark}
\aliascntresetthe{remark}

\newaliascnt{notation}{theorem}
\newtheorem{notation}[notation]{Notation}
\aliascntresetthe{notation}

\newaliascnt{construction}{theorem}

\aliascntresetthe{construction}

\newaliascnt{assumption}{theorem}

\aliascntresetthe{assumption}

\newtheorem*{remark*}{Remark}
\newtheorem*{terminology*}{Terminology}
\newtheorem*{interpretation*}{Interpretation}
\newtheorem*{definition*}{Definition}
\newtheorem*{conjecture*}{Conjecture}
\newtheorem*{notation*}{Notation}
\newtheorem*{convention*}{Convention}

\theoremstyle{remark}

\usepackage[nameinlink,capitalise,noabbrev]{cleveref} % \Cref{prop:1.4} prints ``Proposition 1.4''

\crefname{theorem}{Theorem}{Theorems}
\Crefname{theorem}{Theorem}{Theorems}

\crefname{lemma}{Lemma}{Lemmas}
\Crefname{lemma}{Lemma}{Lemmas}

\crefname{proposition}{Proposition}{Propositions}
\Crefname{proposition}{Proposition}{Propositions}

\crefname{corollary}{Corollary}{Corollaries}
\Crefname{corollary}{Corollary}{Corollaries}

\crefname{conjecture}{Conjecture}{Conjectures}
\Crefname{conjecture}{Conjecture}{Conjectures}

\crefname{question}{Question}{Questions}
\Crefname{question}{Question}{Questions}

\crefname{example}{Example}{Examples}
\Crefname{example}{Example}{Examples}

\crefname{variant}{Variant}{Variants}
\Crefname{variant}{Variant}{Variants}

\crefname{warning}{Warning}{Warnings}
\Crefname{warning}{Warning}{Warnings}

\crefname{recollection}{Recollection}{Recollections}
\Crefname{recollection}{Recollection}{Recollections}

\crefname{movable}{Item to be moved}{Items to be moved}
\Crefname{movable}{Item to be moved}{Items to be moved}

\crefname{definition}{Definition}{Definitions}
\Crefname{definition}{Definition}{Definitions}

\crefname{remark}{Remark}{Remarks}
\Crefname{remark}{Remark}{Remarks}

\crefname{notation}{Notation}{Notations}
\Crefname{notation}{Notation}{Notations}

\crefname{construction}{Construction}{Constructions}
\Crefname{construction}{Construction}{Constructions}

\crefname{assumption}{Assumption}{Assumptions}
\Crefname{assumption}{Assumption}{Assumptions}

\numberwithin{equation}{section}

\makeatletter
  \def\subsection{\@startsection{subsection}{1}%
  \z@{.7\linespacing\@plus\linespacing}{.5\linespacing}%
  {\normalfont\bfseries\centering}}% NEW
\makeatother

\let\oldtocsection=\tocsection
\let\oldtocsubsection=\tocsubsection
\let\oldtocsubsubsection=\tocsubsubsection
\renewcommand{\tocsection}[2]{\hspace{0em}\oldtocsection{#1}{#2}}
\renewcommand{\tocsubsection}[2]{\hspace{1em}\oldtocsubsection{#1}{#2}}
\renewcommand{\tocsubsubsection}[2]{\hspace{2em}\oldtocsubsubsection{#1}{#2}}

\calclayout

\newcommand{\presheaves}{\mathcal{P}}

\DeclareMathOperator{\fil}{fil}
\newcommand{\Hrm}{\mathrm{H}}
\DeclareMathOperator{\Fil}{Fil} % filtered spectra 
\newcommand{\compFil}{\Fil^{c}} % complete filtered spectra 
\newcommand{\compbiFil}{\mathrm{BiFil}^{c}} % complete bifiltered spectra 
\newcommand{\BiFil}{\mathrm{BiFil}}
\newcommand{\Poly}{\mathrm{Poly}}
\newcommand{\AntHKR}{\mathrm{Ant}}
\newcommand{\RakHKR}{\mathrm{Rak}} 
\newcommand{\BLHKR}{\mathrm{BhLu}} 

\DeclareMathOperator{\gr}{gr}
\DeclareMathOperator{\THH}{THH}
\DeclareMathOperator{\HH}{HH}

\mathchardef\mhyphen="2D

%% file: macros.tex
\newcommand{\Dc}{\mathcal{D}}

\newcommand{\Zbf}{\mathbf{Z}}

\newcommand{\op}{\mathrm{op}}

\newcommand{\Fun}{\mathrm{Fun}}

\newcommand{\Map}{\mathrm{Map}}

\newcommand{\Mod}{\mathrm{Mod}}

\newcommand{\ev}{\mathrm{ev}}
\newcommand{\pev}{\mathrm{pev}}

\usepackage{relsize}
\usepackage[bbgreekl]{mathbbol}
\DeclareSymbolFontAlphabet{\mathbb}{AMSb} %to ensure that the meaning of \mathbb does not change
\DeclareSymbolFontAlphabet{\mathbbl}{bbold}

\newcommand{\Dec}{\mathrm{Dec}}

\newcommand{\CAlg}{\mathrm{CAlg}}

\newcommand{\down}{\downarrow}